\documentclass[12pt]{amsart}

        \usepackage{latexsym}
        \usepackage{amssymb}
        \usepackage{amsmath}
        \usepackage{amsfonts}
        \usepackage{amsthm}
        \usepackage[hypertexnames=false]{hyperref}
        \ifpdf
          \usepackage[all,knot,poly,2cell]{xy}
        \else
          \usepackage[all,knot,poly,2cell,dvips]{xy}
        \fi
        \newdir{(}{{}*!/-5pt/@^{(}}
             \UseAllTwocells
        \usepackage{xspace}
        \usepackage{comment}
        \usepackage{fixltx2e}
        \usepackage[toc,page]{appendix}
        \usepackage{url}
        \usepackage{color}
        \usepackage{fixltx2e}
        \usepackage{enumitem}
        \setlist[enumerate]{font=\normalfont}
        \usepackage{mathtools} 

        \usepackage{eucal}
     \usepackage{mdframed}

\usepackage[left=1.5in, right=1.5in, top = 1.2 in, bottom = 1.2in]{geometry}

\newcommand{\spref}[1]{\href{http://stacks.math.columbia.edu/tag/#1}{#1}}

        \theoremstyle{plain}
        \newtheorem{theorem}{Theorem}[section]
        \newtheorem{corollary}[theorem]{Corollary}
        \newtheorem{lemma}[theorem]{Lemma}
        \newtheorem{proposition}[theorem]{Proposition}

\newmdtheoremenv{blackbox}{Black Box}

        \theoremstyle{definition}
        \newtheorem{definition}[theorem]{Definition}
        \newtheorem{example}[theorem]{Example}
        
                \newtheorem{question}[theorem]{Question}

        \theoremstyle{remark}
        \newtheorem{remark}[theorem]{Remark}

\newcommand{\GG}{{\mathbb G}}

\newcommand{\PP}{{\mathbb P}}
\newcommand{\ZZ}{{\mathbb Z}}
\renewcommand{\AA}{{\mathbb A}}

\newcommand{\oh}{\mathcal O}

\newcommand{\cI}{\mathcal{I}}
\newcommand{\cE}{\mathcal{E}}
\newcommand{\cF}{\mathcal{F}}

\newcommand{\cT}{\mathcal{T}}
\newcommand{\cX}{\mathcal{X}}
\newcommand{\cY}{\mathcal{Y}}
\newcommand{\cU}{\mathcal{U}}
\newcommand{\cV}{\mathcal{V}}
\newcommand{\cZ}{\mathcal{Z}}
\newcommand{\cW}{\mathcal{W}}
\newcommand{\cC}{\mathcal{C}}
\newcommand{\cM}{\mathcal{M}}

\DeclareMathOperator{\coker}{coker}
\DeclareMathOperator{\im}{im} \DeclareMathOperator{\Hom}{Hom}

\newcommand{\fM}{\mathfrak{M}}

\DeclareMathOperator{\Spec}{Spec}

\DeclareMathOperator{\Isom}{Isom}
\DeclareMathOperator{\Mor}{Mor}

\DeclareMathOperator{\id}{id}

\DeclareMathOperator{\Aut}{Aut}
\DeclareMathOperator{\GL}{GL}
\DeclareMathOperator{\Gr}{Gr}

\DeclareMathOperator{\Def}{Def}
\DeclareMathOperator{\Ob}{Ob}
\DeclareMathOperator{\red}{red}
\DeclareMathOperator{\Frame}{Frame}

\newcommand{\AffSch}{\mathsf{AffSch}}
\newcommand{\Sch}{\mathsf{Sch}}
\newcommand{\Rep}{\mathsf{Rep}}
\newcommand{\Vect}{\mathsf{Vect}}
\newcommand{\Sets}{\mathsf{Sets}}
\renewcommand{\ss}{\mathrm{ss}}
\DeclareMathOperator{\Br}{Br}

\renewcommand{\bar}{\overline}
\renewcommand{\tilde}{\widetilde}
\newcommand{\colim}{\varinjlim}
\newcommand{\ilim}{\varprojlim}
\newcommand{\fm}{\mathfrak{m}}
\newcommand{\fp}{\mathfrak{p}}
\newcommand{\tensor}{\otimes}
\renewcommand{\H}{\mathrm{H}}
\newcommand{\epf}{\qed \vspace{+10pt}}
\renewcommand{\hat}{\widehat}
\newcommand{\co}{\colon}
\newcommand{\iso}{\stackrel{\sim}{\to}}
\newcommand{\dlim}{\colim}

\newcommand{\AR}{\mathsf{AR}}
\newcommand{\inj}{\hookrightarrow} 
\newcommand{\FiniteA}{\mathsf{Finite} \, A\text{-}\mathsf{mod}}
\newcommand{\Coh}{\mathsf{Coh}} 
\newlength{\arrow}
\renewcommand{\o}{\mathrm{o}}

\newcommand{\gitq}{/\!\!/}
\newcommand*{\myrightarrow}[1]{\xrightarrow{\mathmakebox[\arrow]{#1}}}
\mathchardef\mhyphen="2D

\begin{document}

\title{Artin algebraization and quotient stacks}
\author[Alper]{Jarod Alper}
\address[Alper]{Mathematical Sciences Institute\\
Australia National University\\
Canberra, ACT 0200} \email{jarod.alper@math.anu.edu.au}

\begin{abstract}
This article contains a slightly expanded version of the lectures given by the author at the summer school  ``Algebraic stacks and related topics" in Mainz, Germany from August 31 to September 4, 2015.  The content of these lectures is purely expository and consists of two main goals.  First, we provide a treatment of Artin's approximation and algebraization theorems following the ideas of Conrad and de Jong which rely on a deep desingularization result due to N\'eron and Popescu.  Second, we prove that under suitable hypotheses, algebraic stacks are \'etale locally quotients stacks in a neighborhood of a point with a linearly reductive stabilizer.
\end{abstract}
\maketitle
\tableofcontents

\section*{Introduction}
The goal of these lectures is twofold:  
\begin{enumerate}
	\item Discuss Artin's approximation and algebraization theorems.  We will in fact prove that both theorems follow from a deep desingularization theorem due to N\'eron and Popescu.  This approach follows the ideas of Conrad and de Jong.  
	\item Prove that ``algebraic stacks with linearly reductive stabilizers at closed points are \'etale locally quotient stacks."  See Theorem \ref{T:general} for a precise statement.  This theorem was established by the author, Hall and Rydh in \cite{ahr}.
\end{enumerate}

These two goals are connected in the sense that Artin's theorems will be one of the key ingredients in establishing the main theorem expressed in the second goal.  In fact, the proof of Theorem \ref{T:general} will rely on an equivariant generalization of Artin's algebraization theorem.    Perhaps more importantly though, both these goals shed light on the local structure of algebraic stacks.  Artin's approximation and algebraization theorems together with Artin's criterion for algebraicity instruct us on how we should think of the local structure of algebraic stacks.  The main theorem in the second goal yields a more refined and equivariant understanding of the local structure of algebraic stacks in the case that the stabilizers of the closed points have linearly reductive stabilizers, a property that is often satisfied for algebraic stacks appearing in moduli theory. 

\bigskip

\noindent {\bf Acknowledgements.}  We thank Ariyan Javanpeykar and Ronan Terpereau for organizing the summer school  ``Algebraic stacks and related topics" and we thank the attentive audience of this summer school for providing worthwhile feedback.  In particular, we thank Pieter Belmans for providing a number of corrections of the first draft of these notes.  Finally, we thank Jack Hall and David Rydh for extensive and useful suggestions on the content of these lectures.

\newpage 

\section*{Lecture 1: Artin approximation}
\refstepcounter{section}

The goal of this first lecture is to discuss Artin approximation and provide a few applications.  While we will not give a complete and self-contained proof of Artin approximation, we will show how it follows from N\'eron--Popescu desingularization, which is a very deep and difficult result.  

Throughout these lectures, $k$ will denote an algebraically closed field.  Although almost every statement we make can be generalized to an arbitrary base scheme $S$ with the arguments being essentially identical, it is nevertheless my belief that working over a fixed algebraically closed field $k$ makes the geometric content of the statements more transparent.  Once the geometric content is digested over the field $k$, the motivated reader will have no problem stating and proving the analogous statements over an arbitrary base scheme.

\subsection{N\'eron--Popescu desingularization}

\begin{definition}
A ring homomorphism $A \to B$ of noetherian rings is called {\it geometrically regular} if $A \to B$ is flat and for every prime ideal $\fp \subset A$ and every finite field extension $k(\fp) \to k'$ (where $k(\fp) = A_{\fp} / \fp$), the fiber $B \tensor_A k'$  is regular.
\end{definition}

\begin{remark} It is important to note that $A \to B$ is {\it not} assumed to be of finite type.  In the case that $A \to B$ is a ring homomorphism (of noetherian rings) of finite type, then $A \to B$ is geometrically regular if and only if $A \to B$ is smooth (i.e. $\Spec B \to \Spec A$ is smooth).  
\end{remark}

\begin{remark}  It can be shown that it is equivalent to require the fibers $B \tensor_A k'$ to be regular only for {\it inseparable} field extensions $k(\fp) \to k'$.  In particular, in characteristic $0$, $A \to B$ is geometrically regular if it is flat and for every prime ideal $\fp \subset A$, the fiber $B \tensor_A k(\fp)$ is regular.
\end{remark}

We will accept the following result as a black box.  The proof is difficult.

\medskip

\begin{blackbox}[N\'eron--Popescu desingularization] \label{bb1}
Let $A \to B$ be a ring homomorphism of noetherian rings.  Then $A \to B$ is geometrically regular if and only if $B = \colim B_{\lambda}$ is a direct limit of smooth $A$-algebras.
\end{blackbox}

\begin{remark}
This was result was proved by N\'eron in \cite{neron} in the case that $A$ and $B$ are DVRs and in general by Popescu in \cite{popescu1}, \cite{popescu2}, \cite{popescu3}.  We recommend \cite{swan} and \cite[Tag~\spref{07GC}]{stacks-project} for an exposition on this result.
\end{remark}

\begin{example}
If $l$ is a field and $l^{s}$ denotes its separable closure, then $l \to l^s$ is geometrically regular.  Clearly, $l^s$ is the direct limit of separable field extensions $l \to l'$ (i.e. \'etale and thus smooth $l$-algebras).  If $l$ is a perfect field, then any field extension $l \to l'$ is geometrically regular---but if $l \to l'$ is not algebraic, it is not possible to write $l'$ is a direct limit of \'etale $l$-algebras.
On the other hand, if $l$ is a non-perfect field, then $l \to \bar{l}$ is not geometrically regular as the geometric fiber is non-reduced and thus not regular.
\end{example}

In order to apply N\'eron--Popescu desingularization, we will need the following result, which we will also accept as a black box.  The proof is substantially easier than N\'eron--Popescu's result but nevertheless requires some effort.

\medskip

\begin{blackbox} \label{bb2}
If $S$ is a scheme of finite type over $k$ and $s \in S$ is a $k$-point, then $\oh_{S,s} \to \hat{\oh}_{S,s}$ is geometrically regular.
\end{blackbox}

\begin{remark}  See \cite[IV.7.4.4]{ega} or  \cite[Tag~\spref{07PX}]{stacks-project} for a proof.
\end{remark}

\begin{remark} \label{R:Gring}
A local ring $A$ is called a {\it $G$-ring} if the homomorphism $A \to \hat{A}$ is geometrically regular.  We remark that one of the conditions for a scheme $S$ to be {\it excellent} is that every local ring is a  $G$-ring.  Any scheme that is finite type over a field or $\ZZ$ is excellent.
\end{remark}

\subsection{Artin approximation}
Let $S$ be a scheme and consider a contravariant functor 
$$F \co \Sch/S \to \Sets$$
 where $\Sch/S$ denotes the category of schemes over $S$.  An important example of a contravariant functor is the functor representing a scheme:  if  $X$ is a scheme over $S$, then the {\it functor representing $X$} is:
 \begin{equation} \label{E:representable-functor}
 h_X \co \Sch/S \to \Sets , \qquad (T \to S) \mapsto \Hom_S(T, X).
 \end{equation} 
 We say that $F$ is {\it limit preserving} if for every direct limit $\dlim B_{\lambda}$ of $\oh_S$-algebras $B_{\lambda}$ (i.e. a direct limit of commutative rings $B_{\lambda}$ together with morphisms $\Spec B_{\lambda} \to S$), the natural map
$$\dlim F(\Spec B_{\lambda}) \to F(\Spec \dlim B_{\lambda})$$
is bijective.  This should be viewed as a finiteness condition on the functor $F$.  For instance, the functor $h_X$ from \eqref{E:representable-functor} is limit preserving if and only if $X \to S$ is locally of finite presentation (equivalently locally of finite type if $S$ is noetherian).

Recall that Yoneda's lemma asserts that for an $S$-scheme $T$, there is a natural bijection between $F(T)$ and the set $\Hom(T,F)$ of natural transformation of functors (where we are abusing notation for writing $T$ rather than its representable functor $h_T$).   Moreover, we will consistently abuse notation by conflating objects  $\xi \in F(T)$ and morphisms (i.e. natural transformations of functors) $\xi \co T \to F$.  

\begin{theorem}[Artin approximation] \label{T:approx}
Let $S$ be a scheme of finite type over $k$ and let
$$F \co \Sch/S \to \Sets$$
be a limit preserving contravariant functor.  Let $s \in S$ be a $k$-point and $\hat{\xi} \in F(\Spec \hat{\oh}_{S,s})$.  For any integer $N \ge 0$, there exist an \'etale morphism $(S', s') \to (S, s)$ and an element $\xi' \in F(S')$ such that the restrictions of $\hat{\xi}$ and $\xi'$ to $\Spec(\oh_{S,s} / \fm_s^{N+1})$ are equal (under the identification $\oh_{S,s} / \fm_s^{N+1} \cong \oh_{S',s'} / \fm_{s'}^{N+1}$).
\end{theorem}

\begin{remark}  This was proven in \cite[Cor.~2.2]{artin-approx} in the more general case that $S$ is of finite type over a field or an excellent dedekind domain.  In fact, this theorem holds when the base scheme $S$ is excellent and our proof below works in this generality with only minor modifications.
\end{remark}

\begin{remark}  It is not possible in general to find $\xi' \in F(S')$ restricting to $\hat{\xi}$ or even such that the restrictions of $\xi'$ and $\hat{\xi}$ to $\Spec \oh_{S,s} / \fm_s^{n+1}$ agree for all $n \ge 0$.  For instance, $F$ could be the functor $h_{\AA^1}$ representing the affine line $\AA^1$ and $\hat{\xi} \in \hat{\oh}_{S,s}$ could be a non-algebraic power series.
\end{remark}

\subsection{Alternative formulations of Artin approximation}  

Consider the case that $S = \Spec A$ is an affine scheme of finite type over $k$ and $h_X \co \Sch/S \to \Sets$ is the functor representing the affine scheme $X = \Spec A[x_1, \ldots, x_n]/$ $(f_1, \ldots, f_m)$ of finite type over $S$.  Restricted to the category of affine schemes over $S$ (or equivalently $A$-algebras), the functor is:
$$\begin{aligned}
h_X \co \AffSch / S & \to \Sets \\
	\Spec R & \mapsto \{ a = (a_1, \ldots, a_n) \in R^{\oplus n} \, \mid \, f_i(a) = 0 \text { for all $i$}\}
\end{aligned}
$$
Applying Artin approximation to this functor, we obtain:
\begin{corollary}  \label{C:artin-approx-reformulation}
Let $A$ be a finitely generated $k$-algebra and $\fm \subset A$ be a maximal ideal.  Let $f_1, \ldots, f_m \in A[x_1, \ldots, x_n]$ be polynomials.  Let $\hat{a} = (\hat{a}_1, \ldots, \hat{a}_n) \in \hat{A}_{\fm}$ be a solution to the equations $f_1(x)= \cdots = f_m(x)=0$.  Then for every $N \ge 0$, there exist an \'etale ring homomorphism $A \to A'$, a maximal ideal $\fm' \subset A'$ over $\fm$, and a solution $a' = (a'_1, \ldots, a'_n) \in A'^{\oplus n}$  to the equations $f_1(x)= \cdots = f_m(x)=0$ such that $a' \cong \hat{a} \mod \fm^{N+1}$. \epf
\end{corollary}

\begin{remark}
Although this corollary may seem weaker than Artin approximation, it is not hard to see that it in fact directly implies Artin approximation.   Indeed, writing $S= \Spec A$, we may write $\hat{\oh}_{S,s}$ as a direct limit of finite type $A$-algebras and since $F$ is limit preserving, we can find a commutative diagram
$$\xymatrix{
\Spec \hat{\oh}_{S,s} \ar[d] \ar[rd]^{\hat{\xi}}	& \\
\Spec A[x_1, \ldots, x_n]/(f_1, \ldots, f_m) \ar[r]_{\qquad \qquad \qquad \xi}		& F.
}$$
The vertical morphism corresponds to a solution $\hat{a} = (\hat{a}_1, \ldots, \hat{a}_n) \in \hat{\oh}_{S,s}^{\oplus n}$ to the equations $f_1(x)= \cdots =f_m(x)=0$.  Applying Corollary \ref{C:artin-approx-reformulation} yields the desired \'etale morphism $(\Spec A', s') \to (\Spec A,s)$ and a solution $a'=(a'_1, \ldots, a'_n) \in A'^{\oplus n}$ to the equations $f_1(x)= \cdots =f_m(x)=0$ agreeing with $\hat{a}$ up to order $N$ (i.e. congruent modulo $\fm^{N+1}$).  This induces a morphism
$$\xi' \co \Spec A' \to \Spec A[x_1, \ldots, x_n]/(f_1, \ldots, f_m) \to F$$
which agrees with $\hat{\xi} \co \Spec \hat{\oh}_{S,s} \to F$ to order $N$.
\end{remark}

Alternatively, we can state Corollary \ref{C:artin-approx-reformulation} using henselian rings.
Recall that a local ring $(A, \fm)$ is called {\it henselian} if the following analogue of the implicit function theorem holds:  if $f_1, \ldots, f_n \in A[x_1, \ldots, x_n]$ and $\bar{a} = (\bar{a}_1, \ldots, \bar{a}_n) \in (A/\fm)^{\oplus n}$ is a solution to the equations $f_1(x) = \cdots = f_n(x)=0$ modulo $\fm$ and $\det \big(\frac{\partial f_i}{\partial x_j} (\bar{a})\big)_{i,j=1, \ldots, n} \neq 0$, then there exists a solution $a = (a_1, \ldots, a_n) \in A^{\oplus n}$ to the equations $f_1(x) = \cdots = f_n(x)=0$.   Equivalently, if $(A,\fm)$ is a local $k$-algebra with $A/\fm \cong k$, then $(A, \fm)$ is henselian if every \'etale homomorphism 
$(A, \fm) \to (A', \fm')$ of local rings with $A/\fm \cong A'/\fm'$ is an isomorphism.
Also, if $S$ is a scheme and $s \in S$ is a point, one defines the {\it henselization $\oh_{S,s}^h$ of $S$ at $s$} to be
$$\oh_{S,s}^{h} = \dlim_{(S',s') \to (S,s) } \Gamma(S', \oh_{S'})$$
where the  direct limit is over all \'etale morphisms $(S',s') \to (S,s)$.   In other words, $\oh_{S,s}^h$ is the local ring of $S$ at $s$ in the \'etale topology.

\begin{corollary} \label{C:henselian}
Let $(A, \fm)$ be a local henselian ring which is the henselization of a scheme of finite type over $k$ at a $k$-point.\footnote{More generally, this statement is true for any local henselian $G$-ring (see Remark \ref{R:Gring}) and, in particular, any local henselian excellent ring.}  Let $f_1, \ldots, f_m \in A[x_1, \ldots, x_n]$.  Suppose that $\hat{a} = (\hat{a}_1, \ldots, \hat{a}_n) \in \hat{A}^{\oplus n}$ is a solution to the equations $f_1(x) = \cdots = f_m(x)=0$.  For any integer $N \ge 0$, there exists a solution $a = (a_1, \ldots, a_n) \in A^{\oplus n}$  to the equations $f_1(x) = \cdots = f_m(x)=0$ such that $\hat{a} \cong a \mod \fm^{N+1}$.
\end{corollary}

\subsection{A first application of Artin approximation}
The next corollary states an important fact which you may have taken for granted: if two schemes are formally isomorphic, then they are isomorphic in the \'etale topology.  First, we recall that if $X \to Y$ is a morphism of schemes of finite type over $k$ and $x \in X$ is a $k$-point, then $X \to Y$ is \'etale at $x$ if and only if the induced homomorphism $\hat{\oh}_{Y,f(x)} \to \hat{\oh}_{X,x}$ of completions is an isomorphism.

\begin{corollary}  \label{C:etale-local} Let $X_1, X_2$ be schemes of finite type over $k$.  Suppose $x_1 \in X_1, x_2 \in X_2$ are $k$-points such that $\hat{\oh}_{X_1,x_1} \cong \hat{\oh}_{X_2, x_2}$.  Then there exists a common \'etale neighborhood 
\begin{equation} \label{E:common-etale}
\begin{split}
\xymatrix{
		& (X',x') \ar[rd] \ar[ld] \\
(X_1, x_1)		& 	& (X_2, x_2) \, .
}
\end{split}
\end{equation}
\end{corollary}

\begin{proof}
The functor
$$F \co \Sch/X_1 \to \Sets, \quad (T \to X_1) \mapsto \Hom(T, X_2) $$
is limit preserving as it can be identified with the representable functor \linebreak  $\Hom_{X_1}(-, X_2 \times X_1)$ corresponding to the finite type morphism $X_2 \times X_1 \to X_1$.   The isomorphism $\hat{\oh}_{X_1, x_1}  \cong \hat{\oh}_{X_2, x_2}$ provides an element of  $F(\Spec \hat{\oh}_{X_1, x_1})$.  By applying Artin approximation with $N=1$, we obtain a
diagram as in \eqref{E:common-etale} with $X' \to X_1$ \'etale at $x'$ and such that $\oh_{X_2, x_2} / \fm_{x_2}^2 \to \oh_{X',x'} / \fm_{x'}^2$ is an isomorphism.
But as $\hat{\oh}_{X',x'}$ is abstractly isomorphic to $\hat{\oh}_{X_2, x_2}$, we can conclude that the homomorphism $\hat{\oh}_{X_2, x_2} \to \hat{\oh}_{X',x'}$  induced by $X' \to X_2$ is an isomorphism.\footnote{\label{F:local}The fact that we are using here is that if $(A, \fm)$ is a local noetherian ring and $\phi \co A \to A$ is a local homomorphism such that the induced map $A/\fm^2 \to A/\fm^2$ is an isomorphism, then $\phi$ is an isomorphism.  Indeed, one can use Nakayama's lemma to show that the inclusion $\phi(\fm) \subset \fm$ is also surjective and then Nakayama's lemma again to show that $\phi \co A \to A$ is surjective.  (See also Lemma \ref{L:surjective} or \cite[Lem.~II.7.4]{hartshorne}.)  Finally, we use the fact a surjective endomorphism of noetherian rings is necessarily an isomorphism.}
\end{proof}

\subsection{Proof of Artin approximation}

\begin{proof}[Proof of Artin approximation (Theorem \ref{T:approx})]
By Black Box \ref{bb2}, the morphism $\oh_{S,s} \to \hat{\oh}_{S,s}$
is geometrically regular.  By Black Box \ref{bb1} (N\'eron--Popescu desingularization), $\hat{\oh}_{S,s} = \dlim B_{\lambda}$ is a direct limit of smooth $\oh_{S,s}$-algebras.  Since $F$ is limit preserving, there exist $\lambda$, a factorization $\oh_{S,s} \to  B_{\lambda} \to \hat{\oh}_{S,s}$  and an element $\xi_{\lambda} \in F(\Spec B_{\lambda})$ whose restriction to $F(\Spec \hat{\oh}_{S,s})$ is $\hat{\xi}$.

Let $B = B_{\lambda}$ and $\xi = \xi_{\lambda}$.  Geometrically, we have a commutative diagram
$$\xymatrix{
\Spec \hat{\oh}_{S,s}   \ar[r]^g \ar@/^4ex/[rr]^{\hat{\xi}}   \ar[rd]			& \Spec B \ar[d] \ar[r]^{\xi}	& F\\
		& \Spec \oh_{S,s}
}$$
where $\Spec B \to \Spec \oh_{S,s}$ is smooth.
We claim that we can find a commutative diagram
\begin{equation} \label{D:desired}
\begin{split}
\xymatrix{
 S'  \ar[rd] \ar@{^(->}[r]	  & \Spec B \ar[d] 	 \\	
& \Spec \oh_{S,s}
}
\end{split}
\end{equation}
where $S' \hookrightarrow \Spec B$ is a closed immersion, $(S', s') \to (\Spec \oh_{S,s}, s)$ is \'etale, and the composition $\Spec \oh_{S,s}/ \fm_s^{N+1} \to S' \to \Spec B$ agrees with the restriction of $g\co \Spec \hat{\oh}_{S,s} \to \Spec B$.\footnote{This is where the approximation occurs.  It is not possible to find a morphism $S' \to \Spec B \to \Spec \oh_{S,s}$ which is \'etale at a point $s'$ over  $s$ such that the composition $\Spec \hat{\oh}_{S,s} \to S' \to \Spec B$ is equal to $g$.  
}

    To see this,  
observe that the $B$-module of relative differentials $\Omega_{B / \oh_{S,s}}$ is locally free.  After shrinking $\Spec B$ around the image of the closed point under $\Spec \hat{\oh}_{S,s} \to \Spec B$, we may assume $\Omega_{B / \oh_{S,s}}$ is free with basis $d b_1, \ldots, db_n$. This induces a homomorphism $\oh_{S,s}[x_1, \ldots, x_n] \to B$ defined by $x_i \mapsto b_i$ and provides a factorization
$$\xymatrix{
 \Spec B \ar[r] \ar[d]		& \AA^n_{\oh_{S,s}} \ar[ld] \\
 \Spec \oh_{S,s}
}$$
where $\Spec B \to \AA^n_{\oh_{S,s}}$ is \'etale.  We may choose a lift of the composition
$$
\oh_{S,s}[x_1, \ldots, x_n] \to B \to \hat{\oh}_{S,s} \to \oh_{S,s} / \fm_s^{N+1}
$$
to a morphism $\oh_{S,s}[x_1, \ldots, x_n] \to \oh_{S,s}$.  This gives a section $s \co \Spec \oh_{S,s} \to \AA^n_{\oh_{S,s}}$ and we define $S'$ as the fibered product
$$\xymatrix{
S' \ar@{^(->}[d]	 \ar[r]		& \Spec \oh_{S,s} \ar@{^(->}[d]^{s} \\
\Spec B \ar[r]\ar@{}[ur]|\square				& \AA^n_{\oh_{S,s}}.
}$$
This gives the desired Diagram \ref{D:desired}.  The composition $\xi' \co S' \to \Spec B \xrightarrow {\xi}  F$ is an element which agrees with $\hat{\xi}$ up to order $N$.

 By ``standard direct limit" methods, we may ``smear out" the \'etale morphism $(S', s') \to (\Spec \oh_{S,s}, s)$ and the element $\xi' \co S' \to F$ to find an \'etale morphism $(S'', s'') \to (S,s)$ and an element $\xi'' \co S'' \to F$ agreeing with $\hat{\xi}$ up to order $N$.  Since this may not be standard for everyone, we spell out the details.  Let $\Spec A \subset S$ be an open affine containing $s$.   We may write $S' = \Spec A'$ and $A' = \oh_{S,s}[y_1, \ldots, y_n]/ (f'_1, \ldots, f'_m)$.  As $\oh_{S,s} = \dlim_{g \notin \fm_s} A_g$, we can find an element $g \notin \fm_s$ and elements $f''_1, \ldots, f''_m \in A_g[y_1, \ldots, y_n]$ restricting to $f'_1, \ldots, f'_m$.  Let $S'' = \Spec A_g[y_1, \ldots, y_n] / (f''_1, \ldots, f''_m)$ and $s'' \in S''$ be the image of $s'$ under $S' \to S''$.  Then $S'' \to S$ is \'etale at $s''$.  As $A' = \dlim_{h \notin \fm_s} A_{hg}[y_1, \ldots, y_n] / (f'_1, \ldots, f'_m)$ and $F$ is limit preserving, we can, after replacing $g$ with $hg$, find an element $\xi'' \in F(S'')$ restricting to $\xi'$ and, in particular, agreeing with $\hat{\xi}$ up to order $N$. Finally, we shrink $S''$ around $s''$ so that $S'' \to S$ is \'etale everywhere.
 \end{proof}

\subsection{Categories fibered in groupoids}
We now shift in the direction of stacks by introducing categories fibered in groupoids.

 Let $S$ be a scheme.  Let $\cX$ be a category and $p\co \cX \to \Sch/S$ be a functor.   We visualize this data as
$$\xymatrix{
\cX \ar[dd]^p &	 \xi_1 \ar@{|->}[dd] \ar[rd]^{\beta} \ar[r]^{\alpha}	& \xi_2\\
				 & & \eta_2\\
\Sch/S	&  T_1 \ar[r]^f & T_2
}$$
where the greek letters $\xi_1$ and $\xi_2,\eta_2$ are objects in $\cX$ over the $S$-schemes $T_1$ and $T_2$, respectively, and the morphisms $\alpha, \beta$ are over $f$.

The main motivation for introducing categories fibered in groupoids rather than simply working with functors as above is to have a nice framework to handle moduli spaces parameterizing objects with automorphisms (see Example \ref{E:curves}).

\begin{definition} \label{D:cfg}
A {\it category fibered in groupoids over $S$} is a category $\cX$ together with a functor $p \co \cX \to \Sch/S$ such that:
\begin{enumerate}
	\item (``existence of pullbacks") For every morphism of $S$-schemes $T_1 \to T_2$ and object $\xi_2$ of $\cX$ over $T_2$, there exists an object $\xi_1$ over $T_1$ completing the diagram
$$\xymatrix{
\xi_1 \ar@{-->}^{\alpha}[r]	\ar@{|-->}[d]& \xi_2 \ar@{|->}[d] \\
T_1 \ar[r]^f	& T_2
}$$
where the morphism $\alpha$ is over $f$ (i.e. $p(\alpha) = f$).
\item (``uniqueness and composition of pullbacks'')  For all diagrams
$$\xymatrix{
\xi_1 \ar[rrd] \ar@{-->}[rd] \ar@{|->}[dd] 	\\
		&\xi_2 \ar[r]\ar@{|->}[dd] 	&\xi_3 \ar@{|->}[dd] \\
T_1 \ar[rrd] \ar[rd]	\\
		& T_2 \ar[r]	&T_3
}$$
there exists a unique arrow $\xi_1 \to \xi_2$ over $T_1 \to T_2$ filling in the diagram.
\end{enumerate}
We will often simply write $\cX$ for a category fibered in groupoids over $S$ where it is implicitly understood that part of the data is the functor $p \co \cX \to \Sch/S$.
\end{definition}

\begin{remark}  Axiom (2) above implies that the pullback in Axiom (1) is unique up to unique isomorphism.  Often we write $\xi_2|_{T_1}$ to indicate a \emph{choice} of pullback of $\xi_2$ under $T_1 \to T_2$.  
\end{remark}

If $\cX$ is a category fibered in groupoids over $\Sch/S$ and $T$ is an $S$-scheme, we define the {\it fiber category over $T$}, denoted by $\cX(T)$, as the category whose objects are objects of $\cX$ over $T$ and whose morphisms are morphisms of $\cX$ over the identity morphism $\id_T \co T \to T$.  Axioms (1) and (2) imply that $\cX(T)$ is a groupoid, i.e.  all morphisms in $\cX(T)$ are isomorphisms.  

\begin{example}(Functors as categories fibered in groupoids) \label{E:functor-as-cfg}

A contravariant functor $F \co \Sch/S \to \Sets$ may be viewed as a category fibered in groupoids over $\Sch/S$ as follows.  Let $\cX_F$ be the category whose objects are pairs $(T,\xi)$, where $T$ is an $S$-scheme and $\xi \in F(T)$.  A morphism $(T, \xi)\to (T',\xi')$ in the category $\cX_F$ is given by a morphism $f \co T \to T'$ such that the pullback of $\xi'$ via $f$ is $\xi$ (i.e. $F(f)(\xi')=\xi$).  The functor $ \cX_F \to \Sch/S$ takes $(T, \xi)$ to $T$ and a morphism $(T, \xi) \to (T', \xi')$ to $T \to T'$.  It is easy to see that $\cX_F$ is a category fibered in groupoids over $S$.

In this case, pullbacks (as defined in Definition \ref{D:cfg}(1)) are unique and each fiber category $\cX_F(T)$ is a setoid, i.e. a set with identity morphisms.
\end{example}

\begin{example} (The moduli space $\cM_g$ of smooth curves) \label{E:curves}

We define $\cM_g$ as the category of pairs $(T, \cC \to T)$ where $T$ is a scheme over $k$ and $\cC \to T$ is a smooth proper morphism such that every geometric fiber is a connected genus $g$ curve.  A morphism $(T, \cC\to T) \to (T', \cC' \to T')$ is the data of a cartesian diagram
$$\xymatrix{
\cC \ar[d] \ar[r]		& \cC' \ar[d] \\
T \ar[r] \ar@{}[ur]|\square			& T.
}$$
The functor $\cM_g \to \Sch/S$ takes $(T, \cC \to T)$ to $T$ and a diagram as above to the morphism $T \to T'$.  It is easy to see that $\cM_g$ is a category fibered in groupoids over $k$.
If $C$ is a smooth curve over $k$, note that the morphisms from $C$ to $C$ in the fiber category $\cM_g(k)$ correspond to automorphisms of $C$.
\end{example}

We say that a category fibered in groupoids $\cX$ over $S$  is  {\it limit preserving} if for every direct limit $\dlim B_{\lambda}$ of $\oh_S$-algebras, the natural functor
$$\dlim \cX(\Spec B_{\lambda}) \to \cX(\Spec \dlim B_{\lambda})$$
is an equivalence of categories.

As before, we have a Yoneda's lemma which asserts that if $T$ is an $S$-scheme, there is a natural equivalence of categories between $\cX(T)$ and the category $\Hom(T, \cX)$ of morphisms $T \to \cX$ of categories fibered in groupoids over $S$.\footnote{Here $T$ is considered as the corresponding category fibered in groupoids (for instance by first considering the corresponding functor as in \eqref{E:representable-functor} and then the corresponding category fibered in groupoids via Example \ref{E:functor-as-cfg}).  A morphism $\cX_1 \to \cX_2$ of categories fibered in groupoids over $S$ is a functor $\cX_1 \to \cX_2$ such that the diagram
$$\xymatrix{
\cX_1 \ar[r] \ar[rd]	& \cX_2 \ar[d] \\
				& \Sch/S
}$$
strictly commutes.}
We now restate Artin's approximation using categories fibered in groupoids rather than functors. 

\begin{theorem} [Groupoid version of Artin approximation] \label{T:approx-groupoid} 
Let $S$ be a scheme of finite type over $k$ and $\cX$ be a limit preserving category fibered in groupoids over $S$.   Let $s \in S$ be a $k$-point and $\hat{\xi}$ an object of $\cX$ over $\Spec \hat{\oh}_{S,s}$.  For any integer $N \ge 0$, there exist an \'etale morphism $(S', s') \to (S, s)$ and an element $\xi'$ of $\cX$ over $S'$ such that $\hat{\xi}|_{\Spec(\oh_{S,s} / \fm_s^{N+1})}$ and  $\xi'|_{\Spec( \oh_{S',s'} / \fm_{s'}^{N+1})}$ are isomorphic (under the identification $\oh_{S,s} / \fm_s^{N+1} \cong \oh_{S',s'} / \fm_{s'}^{N+1}$).
\end{theorem}

\begin{proof}
The argument for the functorial case of Artin's approximation goes through essentially without change.  Alternatively, it can be seen to follow directly from the functorial version by considering the underlying functor $F \co \Sch/S \to \Sets$ where for an $S$-scheme $T$, the set $F(T)$ is the set of isomorphism classes of $\cX(T)$.
\end{proof}

\newpage

\section*{Lecture 2: Artin algebraization}
\refstepcounter{section}

The goal of this lecture is to prove Artin algebraization.  
We follow the ideas of Conrad and de Jong from \cite{conrad-dejong} (see also \cite[Tag~\spref{07XB}]{stacks-project}).  Namely, we will show how Artin approximation (Theorem \ref{T:approx}) implies a stronger approximation result (Theorem \ref{T:conrad-dejong}), which we refer to as Conrad--de Jong approximation, and then show how this implies Artin algebraization (Theorem \ref{T:algebraization}).  The logical flow of implications is:
$$\begin{array}{c}
\framebox{N\'eron--Popescu desingularization} \vspace{.1cm}	\\
\Big\Downarrow  \vspace{.1cm} \\
\framebox{Artin approximation} 	\vspace{.1cm}	\\
\Big\Downarrow  \vspace{.1cm} \\
\framebox{Conrad--de Jong approximation} 	\vspace{.1cm}	\\
\Big\Downarrow  \vspace{.1cm} \\
\framebox{Artin algebraization} 	
\end{array}$$

\subsection{Conrad--de Jong approximation}
In Artin approximation the initial data is an object over a complete local noetherian $k$-algebra $\hat{\oh}_{S,s}$ which is assumed to be the completion of a finitely generated $k$-algebra at a maximal ideal.  We will now see that a similar approximation result still holds if this latter hypothesis is dropped where one approximates {\it both} the complete local ring and the object over this ring.  

Recall first that if $(A, \fm)$ is a local ring and $M$ is an $A$-module, then the {\it associated graded module of $M$} is defined as $\Gr_{\fm} M = \bigoplus_{n \ge 0} 
\fm^n M / \fm^{n+1} M$.

 \begin{theorem}[Conrad--de Jong approximation]\label{T:conrad-dejong}
 Let $\cX$ be a limit preserving category fibered in groupoids over $k$.  Let $(R, \fm)$ be a complete local noetherian $k$-algebra and let $\hat{\xi}$ be an object of $\cX$ over $\Spec R$.  Then for every integer $N \ge 0$, there exist
\begin{enumerate}
	\item an affine scheme $\Spec A$ of finite type over $k$ and a $k$-point $u \in \Spec A$, 
	\item an object $\xi_A$ of $\cX$ over $\Spec A$,
	\item an isomorphism $\alpha_{N+1} \co R / \fm^{N+1} \cong A / \fm_u^{N+1}$,
	\item an isomorphism of $\hat{\xi}|_{\Spec(R/ \fm^{N+1})}$ and  $\xi_A|_{\Spec(A/\fm_u^{N+1})}$ via $\alpha_{N+1}$, and
	\item an isomorphism $\Gr_{\fm}(R) \cong \Gr_{\fm_u}(A)$ of graded $k$-algebras.
\end{enumerate}
 \end{theorem}
 
 \begin{remark}  This was proven in \cite{conrad-dejong} (see also  \cite[Tag~\spref{07XB}]{stacks-project}). 
 \end{remark}

The proof of this theorem will proceed by simultaneously approximating equations and relations defining $R$ and the object $\hat{\xi}$ of $\cX$ over $\Spec R$.
The statements (1)--(4) will be easily obtained as a consequence of Artin approximation.
A nice insight of Conrad and de Jong is that condition (5) can be ensured by Artin approximation, and moreover that this condition suffices to imply the isomorphism of complete local $k$-algebras in Artin algebraization.  As such, condition (5) takes the most work to establish.

\begin{proof}[Proof of Conrad--de Jong approximation (Theorem \ref{T:conrad-dejong})]
Since $\cX$ is limit preserving and $R$ is the direct limit over its finitely generated $k$-subalgebras, we can find an affine scheme $V =\Spec B$ of finite type over $k$ and an object $\xi_V$ of $\cX$ over $V$ together with a 2-commutative diagram
$$
\xymatrix{
\Spec R \ar[r]\ar@/^3ex/[rr]^{\hat{\xi}} & V\ar[r]_{\xi_V} & \cX.
}
$$
Let $v \in V$ be the image of the maximal ideal $\fm \subset R$.    
After adding generators to the ring $B$ if necessary, we can assume that the composition $\hat{\oh}_{V,v} \to R \to R /\fm^2$ is surjective.  A basic fact about complete local noetherian rings (Lemma \ref{L:surjective})  implies that $\hat{\oh}_{V,v} \to R$ is surjective.  
The goal now is to simultaneously approximate over $V$ the equations and relations defining the closed immersion $\Spec R \hookrightarrow \Spec \hat{\oh}_{V,v}$ and the object $\hat{\xi}$.  In order to accomplish this goal, we choose a resolution 
\begin{equation} \label{E:resolution}
\hat{\oh}_{V,v}^{\oplus r}\myrightarrow{\hat{\alpha}} \hat{\oh}_{V,v}^{\oplus s} \myrightarrow{\hat{\beta}} \hat{\oh}_{V,v} \myrightarrow{} R,
\end{equation}
and consider the functor
$$\begin{aligned}
F \co (\Sch/V) & \to {\rm Sets} \\
  (T \to V) 	& \mapsto \big\{ \text{complexes } \oh_T^{\oplus r} \myrightarrow{\alpha} \oh_T^{\oplus s} \myrightarrow{\beta} \oh_T \big\}
\end{aligned}$$
which is easily checked to be limit preserving.  
The resolution in \eqref{E:resolution} yields an element of $F(\Spec \hat{\oh}_{V,v})$.
If we apply Artin approximation (Theorem \ref{T:approx}), we obtain an \'etale morphism $(V' = \Spec B', v') \to (V, v)$ and an
element 
\begin{equation} \label{E:complex-approx}
(B'^{\oplus r} \myrightarrow{\alpha'} B'^{\oplus s} \myrightarrow{\beta'} B') \in F(V')
\end{equation} such that $\alpha', \beta'$ are equal to $\hat{\alpha}, \hat{\beta}$ modulo $\fm^{N+1}$.  

Let $U = \Spec A \inj \Spec B' = V'$ be the closed subscheme defined by $\im \beta'$ and set $u = v' \in U$.
We have an induced object 
 $$\xi_A \co U \hookrightarrow V' \to V  \xrightarrow{\xi_V} \cX$$
 of $\cX$ over $U$.  As $R = \coker \hat{\beta}$ and $A = \coker \beta'$, we have an isomorphism $R / \fm^{N+1} \cong A / \fm_u^{N+1}$ together with an isomorphism of $\hat{\xi}|_{\Spec(R/ \fm^{N+1})}$ and \linebreak  $\xi_A|_{\Spec(A/\fm_u^{N+1})}$.  This gives statements (1)--(4).
 
 We are left to establish statement (5).  First, we provide some motivation for the technical definition (Definition \ref{D:AR}) and lemma (Lemma \ref{L:artin-rees-generalization})  below.  To establish (5), we need to show that there are isomorphisms $\fm^n / \fm^{n+1} \cong \fm_u^n / \fm_u^{n+1}$.
 For $n \le N$, this is guaranteed by the isomorphism $R / \fm^{N+1} \cong A / \fm_u^{N+1}$. 
 On the other hand, for $n \gg 0 $, this can be seen to be a consequence of the Artin--Rees lemma (see the proof of Lemma \ref{L:artin-rees-generalization}(3) below).  We will need to handle the middle range of $n$ and we accomplish this be controlling the constant appearing in the Artin--Rees lemma. 
 
 We now establish statement (5) using Definition \ref{D:AR} and Lemma \ref{L:artin-rees-generalization}.  We first realize that before we applied Artin approximation, we could have increased $N$ in order to guarantee that $(\AR)_N$ holds for $\hat{\alpha}$ and $\hat{\beta}$.  Therefore, we are free to assume this.  Now statement (5) follows directly if we apply Lemma \ref{L:artin-rees-generalization} to the exact complex $\hat{\oh}_{V,v}^{\oplus r}\myrightarrow{\hat{\alpha}} \hat{\oh}_{V,v}^{\oplus s} \myrightarrow{\hat{\beta}} \hat{\oh}_{V,v}$ from \eqref{E:resolution} and the complex 
 $
\hat{\oh}_{V,v}^{\oplus r}\myrightarrow{\hat{\alpha}'} \hat{\oh}_{V,v}^{\oplus s} \myrightarrow{\hat{\beta}'} \hat{\oh}_{V,v}
$
obtained by restricting  \eqref{E:complex-approx} to $F(\Spec \hat{\oh}_{V,v})$.
\end{proof}

\bigskip

\noindent
{\it The Artin--Rees condition. }

\begin{definition}\label{D:AR} Let $(A,\fm)$ be a local noetherian ring. Let $\varphi \co M \to N$ be a morphism of finite $A$-modules.  Let  $c \geq 0$ be an integer.  We say that {\it $(\AR)_c$ holds for $\varphi$} if
\[
\varphi(M)\cap \fm^n N \subset \varphi(\fm^{n-c}M),\quad \forall n\geq c .
\]
\end{definition}

\begin{remark}
The Artin--Rees lemma (see \cite[Prop.~10.9]{atiyah-macdonald} or  \cite[Lem. 5.1]{eisenbud}) implies that   $(\AR)_c$ holds for $\varphi$ if $c$ is sufficiently large.

\end{remark}

The following lemma from \cite[\S3]{conrad-dejong} (see also \cite[Tags 07VD and 07VF]{stacks-project}) is straightforward to prove.  

\begin{lemma}\label{L:artin-rees-generalization}
Let $(A, \fm)$ be a local noetherian ring.  Let
$$L \xrightarrow{\alpha}M \xrightarrow{\beta} N  \hspace{1cm} \text{and} \hspace{1cm} L \xrightarrow{\alpha'} M \xrightarrow{\beta'} N$$
be two complexes of finite $A$-modules.  Let $c$ be a positive integer. Assume
that
\begin{enumerate}
\item[(a)] the first sequence is exact, 
\item[(b)] the complexes are isomorphic modulo $\fm^{c+1}$, and 
\item[(c)] $(\AR)_c$ holds for $\alpha$ and $\beta$.
\end{enumerate}
Then
\begin{enumerate}
\item $(\AR)_c$ holds for $\beta'$,
\item the second sequence is exact, and 
\item there exists an isomorphism 
$ \Gr_{\fm}(\coker \beta) \to  \Gr_{\fm}(\coker \beta')$
 of 
$\Gr_{\fm}(A)$-modules.
\end{enumerate}
\end{lemma}

\begin{remark}
Only conclusion (3) was used in the proof of Conrad--de Jong approximation.  Statements (1) and (2) are included as they serve as convenient steps in the proof of (3).  Indeed, it is the fact that $(\AR)_c$ holds for $\beta$ that implies the containment ``$\subset$" in \eqref{E:key-equality}.  Likewise, the other containment ``$\supset$" will hold once we know $(\AR)_c$ holds for $\beta'$.
\end{remark}

\begin{proof}
We claim that $(\beta')^{-1}(\fm^n N) \subset \alpha'(L) + \fm^{n-c}M$ for all $n \ge c$.  Let $m \in M$ such that $\beta'(m) \in \fm^n N$.  Suppose we can find $l \in L$ such that  $m_0 = m - \alpha'(l) \in \fm^r M$ with $r < n - c$.  We will show that we can modify $l$ in order to increase $r$ by $1$.  Note that
$$\begin{aligned}
\beta(m_0) 	& = \beta'(m_0) + (\beta - \beta')(m_0) \\
				&  = \beta'(m) + (\beta - \beta')(m_0) \\
				& \in \fm^n N + \fm^{r+c+1} N = \fm^{r+c+1} N
\end{aligned}$$
Since $(\AR)_c$ holds for $\beta$, we have $\beta(m_0) \in \beta(\fm^{r+1}M)$.   So we can write $\beta(m_0) = \beta(m_1)$ with $m_1 \in \fm^{r+1}M$.  As the first complex is exact, we may write $m_0 - m_1 = \alpha(l_1)$ for $l_1 \in L$.  Note that $m_0 - m_1 \in \fm^r M$.  We now break the argument into whether or not $r \ge c$.  If $r \ge c$, then as $(\AR)_c$ holds for $\alpha$, we have $m_0 - m_1 \in \alpha(\fm^{r-c}L)$ and we may replace $l_1$ with an element of $\fm^{r-c}L$.  In this case, $(\alpha - \alpha')(l_1) \in \fm^{r-c} \cdot \fm^{c+1} M = \fm^{r+1} M$.  On the other hand, if $r < c$, then we automatically have $(\alpha - \alpha')(l_1) \in \fm^{c+1} M \subset \fm^{r+1} M$.  Therefore,
$$\begin{aligned}
m - \alpha'(l+l_1)  	& = m_0 - \alpha'(l_1) \\
				& = m_1 + (\alpha - \alpha')(l_1)  \in \fm^{r+1} M.
\end{aligned}$$
By induction, this establishes the claim.

For (1), the condition that $\beta'(M) \cap \fm^n N \subset \beta'(\fm^{n-c}M)$ follows directly from the claim.  For (2), observe that the claim coupled with Krull's intersection theorem implies that
$$(\beta')^{-1}(0) = (\beta')^{-1}( \bigcap_n \fm^n N) \subset \bigcap_n \big(\alpha'(L) + \fm^{n-c} M) = \alpha'(L)  $$
which gives exactness of the second sequence.

For (3), for $n \ge 0$, we have
$$\Gr_{\fm}(\coker \beta)_n = \fm^n N / (\fm^{n+1}N + \beta(M) \cap \fm^n N)$$
and a similar description of $\Gr_{\fm}(\coker \beta')_n$.  To obtain an isomorphism \linebreak $ \Gr_{\fm}(\coker \beta) \to  \Gr_{\fm}(\coker \beta')$, it clearly suffices to show that
\begin{equation} \label{E:key-equality}
\fm^{n+1}N + \beta(M) \cap \fm^n N = \fm^{n+1} N + \beta'(M) \cap \fm^n N.
\end{equation}
We first show the containment ``$\subset$".  If $n \le c$, then the statement is clear as $\beta = \beta' \mod \fm^{c+1}$.  If $n > c$, suppose $x\in \beta(M) \cap \fm^n N$.  Since $(\AR)_c$ holds for $\beta$, we may write $x = \beta(m)$ for $m \in \fm^{n-c} M$.  Let $x' = \beta'(m)$.  Then $x-x' = (\beta - \beta')(m) \in \fm^{c+1} \cdot \fm^{n-c} M = \fm^{n+1}N$.  Since $(\AR)_c$ also holds for $\beta'$, by symmetry we obtain the other containment ``$\supset$".
\end{proof}

\begin{remark} Conrad--de Jong approximation directly implies Artin approximation (Theorem \ref{T:approx-groupoid}).  Indeed, let $R = \hat{\oh}_{S,s}$ and $N \ge 1$.  Apply Conrad--de Jong approximation to $\cX \times S$ to obtain a finite type morphism $(U=\Spec A,u) \to (S,s)$, an object $\xi_A$ of $\cX$ over $U$, an isomorphism $\alpha_{N+1} \co \oh_{S,s}/ \fm_s^{N+1} \cong \oh_{U,u}/ \fm_u^{N+1}$ and an isomorphism $\hat{\xi}|_{\Spec(\oh_{S,s}/ \fm_s^{N+1})}$ and  $\xi_A|_{\Spec(\oh_{U,u}/ \fm_u^{N+1})}$ via $\alpha_{N+1}$.  We claim that $U \to S$ is \'etale at $u$.  Since we know that $\hat{\oh}_{S,s} / \fm_s^2 \to \hat{\oh}_{U,u} / \fm_u^2$ is an isomorphism, the induced homomorphism $\hat{\oh}_{S,s} \to \hat{\oh}_{U,u}$ is surjective.  But condition (5) in Conrad--de Jong approximation gives isomorphisms 
$$\fm_s^N /  \fm_s^{N+1} \to 
\fm_u^N /  \fm_u^{N+1} $$
for every $N$.  This implies that $\hat{\oh}_{S,s} \to \hat{\oh}_{U,u}$ is injective and thus an isomorphism.
\end{remark}

\subsection{Artin algebraization}

Artin algebraization has a stronger conclusion than Artin approximation or Conrad--de Jong approximation in that no approximation is necessary.  It guarantees the existence of an object $\xi_A$ over a pointed affine scheme $(\Spec A,u)$ of finite type over $k$ which agrees with the given effective formal deformation $\hat{\xi}$ to all orders.  However, in order to ensure this, it is necessary to impose a further condition on the object $\hat{\xi}$.  This condition is known as formal versality and is extremely natural (see Remarks \ref{R:versal1}-\ref{R:versal3}).

\begin{definition}  \label{D:versal} Let $\cX$ be a category fibered in groupoids over $k$.  Let $R$ be a complete local noetherian $k$-algebra and $x \in \Spec R$ be the closed point.  We say that an object $\hat{\xi}$ of $\cX$ over $\Spec R$ is {\it formally versal at $x$} if for every commutative diagram
\begin{equation} 
\begin{split}\label{D:lifting-versal}
\xymatrix{
\Spec k(x) \ar@{^(->}[r]	& \Spec B \ar[r]  \ar@{^(->}[d]	&  \Spec R \ar[d]^{\hat{\xi}} \\
& \Spec B' \ar[r] \ar@{-->}[ur]			& \cX
} \end{split}
\end{equation}
where $B' \to B$ is a surjection of artinian $k$-algebras, there is a lift $\Spec B' \to \Spec R$ filling in the above diagram.
\end{definition}

\begin{remark}  \label{R:versal1}  In other words, the formal versality of $\hat{\xi}$ means that whenever you have a surjection $B' \to B$  of artinian $\oh_S$-algebras, an object $\eta'$ of $\cX$ over $\Spec B'$, a morphism $\Spec B \to \Spec R$, and an isomorphism $\alpha \co \hat{\xi}|_{\Spec B} \cong \eta'|_{\Spec B}$, there exist a morphism $\Spec B' \to \Spec R$ and an isomorphism \linebreak $\alpha' \co \hat{\eta}|_{\Spec B'} \cong \eta'$ extending $\alpha$.  

Let $\xi_0 = \hat{\xi}|_{\Spec k(x)}$ be the restriction of $\hat{\xi}$ to the closed point. 
Then \linebreak $\eta \co \Spec B \to \cX$ can be viewed as an {\it infinitesimal deformation} of $\xi_0$ and $\eta' \co \Spec B' \to \cX$ a further infinitesimal deformation of $\eta$.  
In colloquial language, the condition of formal versality implies that the family of objects $\hat{\xi} \co \Spec R \to \cX$ contains all infinitesimal deformations of $\xi_0$. 
\end{remark}

\begin{remark}
The lifting criterion in Diagram \eqref{D:lifting-versal} above should remind the reader of the formal lifting property for smooth morphisms.  Indeed, if $f \co X \to Y$ is a morphism of finite type $k$-schemes, then it is a theorem of Grothendieck \cite[IV.17.14.2]{ega} that $f$ is smooth at $x$ if and only if the above lifting criterion holds (replacing of course $\hat{\xi} \co \Spec R \to \cX$ with $f \co X \to Y$).
\end{remark}

\begin{remark} \label{R:versal3}  \label{R:formal-versality}
It is easy to see that the condition of formal versality of \linebreak $\hat{\xi} \co \Spec R \to \cX$ only depends on the compatible family $\{\xi_n = \hat{\xi}|_{\Spec R / \fm^{n+1}}\}$ of restrictions. Therefore, we can extend the above definition to compatible families $\{\xi_n\}$.
\end{remark}

 \begin{theorem}[Artin algebraization] \label{T:algebraization} 
 Let $\cX$ be a limit preserving category fibered in groupoids over $k$.  Let $(R, \fm)$ be a complete local noetherian $k$-algebra, and let $\hat{\xi}$ be a formally versal object of $\cX$ over $\Spec R$.  There exist
\begin{enumerate}
	\item an affine scheme $\Spec A$ of finite type over $k$ and a $k$-point $u \in \Spec A$, 
	\item an object $\xi_A$ of $\cX$ over $\Spec A$,
	\item an isomorphism $\alpha \co R \iso \hat{A}_{\fm_u}$  of $k$-algebras, and 
	\item a compatible family of isomorphisms $\hat{\xi}|_{\Spec R / \fm^{n+1}} \cong \xi_U|_{\Spec A / \fm_u^{n+1}}$ (under the identification $R / \fm^{n+1} \cong A / \fm_{u}^{n+1}$) for $n \ge 0$.
\end{enumerate}
 \end{theorem}

\begin{remark} This was first proven in \cite{artin-algebraization}.
\end{remark}

\begin{remark} \label{R:proof-algebraization-special-case}
In the case that $R$ is known to be the completion of a finitely generated $k$-algebra,
 this theorem can be viewed as an easy consequence of Artin approximation.
Indeed, one applies Artin approximation with $N=1$ and then uses the formal versality condition to obtain compatible ring homomorphisms $R \to A / \fm_u^{n+1}$ and therefore a ring homomorphism 
$R \to \hat{A}_{\fm_u}$ which is an isomorphism modulo $\fm^2$. As $R$ and $\hat{A}_{\fm_u}$ are abstractly isomorphic, the homomorphism $R \to \hat{A}_{\fm_u}$ is an isomorphism (see Footnote \ref{F:local}) and the statement follows.  We leave the details to the reader but remark that this argument is very analogous to the proof of Artin algebraization below.  For the general case, since we don't know $R$ is the completion of a finitely generated $k$-algebra, we apply Conrad--de Jong approximation instead of Artin approximation.
\end{remark}

\begin{proof}[Proof of Artin algebraization (Theorem \ref{T:algebraization})]
Applying Conrad--de Jong approximation with $N=1$, we obtain an affine scheme $\Spec A$ of finite type over $k$ with a $k$-point $u \in \Spec A$, an object $\xi_A$ of $\cX$ over $\Spec A$, an isomorphism $\alpha_2 \co \Spec A / \fm_u^{2} \to \Spec R / \fm^{2}$, an isomorphism  $\iota_2 \co \hat{\xi}|_{\Spec(R/ \fm^{2})} \to \xi_A|_{\Spec(A/\fm_u^{2})}$, and an isomorphism $\Gr_{\fm}(R) \cong \Gr_{\fm_u}(A)$ of graded $k$-algebras.

We claim that $\alpha_2$ and $\iota_2$ can be extended to a compatible family of morphisms $\alpha_{n} \co \Spec A / \fm_u^{n+1} \to \Spec R$  and isomorphisms $\iota_n \co \hat{\xi}|_{\Spec(A/ \fm_{u}^{n+1})} \to$ \linebreak $\xi_A|_{\Spec(A/\fm_u^{n+1})}$.  By induction, suppose we are given $\alpha_{n}$ and $\iota_n$. Since $\hat{\xi}$ is formally versal, there is a lift $\alpha_{n+1} \co \Spec A/\fm^{n+1} \to \Spec R$ filling in the commutative diagram
$$\xymatrix{
\Spec A/\fm_u^{n} \ar[r]^{\alpha_n} \ar[d]		& \Spec R \ar[d]^{\hat{\xi}} \\
\Spec A/\fm_u^{n+1} \ar[r]_{\qquad \xi_A} \ar@{-->}[ur]_{\alpha_{n+1}}	& \cX,
}$$
which establishes the claim.

By taking the limit, we have a homomorphism $\hat{\alpha} \co R \to \hat{A}_{\fm_u}$ which is surjective as $\hat{\alpha}$ restricts to the given isomorphism $R / \fm^2 \to A / \fm_u^2$ (where we have used Lemma \ref{L:surjective}). On the other hand, for each $n$, we know that the $k$-vector spaces $\fm^N  /  \fm^{N+1}$ and  
$\fm_u^N /  \fm_u^{N+1} $
have the same dimension.  This implies that $\hat{\alpha}$ is an isomorphism which finishes the proof.
\end{proof}

In the above arguments, this fact was used several times.

\begin{lemma}  \label{L:surjective}
Let $(A, \fm_A)$ and $(B, \fm_B)$ be local noetherian complete rings.  If $A \to B$ is a local homomorphism such that $A / \fm_A^2 \to B / \fm_B^2$ is surjective, then $A \to B$ is surjective.
\end{lemma}

\begin{proof}
This follows from the following version of Nakayama's lemma for complete local rings $(A,\fm)$:  if $M$ is a (not-necessarily finitely generated) $A$-module such that $\bigcap_k \fm^k M = 0$ and $m_1, \ldots, m_n \in F$ generate $M/\fm M$, then $m_1, \ldots, m_n$ also generate $M$ (see \cite[Exercise 7.2]{eisenbud}).
\end{proof}

\subsection{Algebraic stacks}

We now quickly introduce the notion of {\it stacks} and {\it algebraic stacks}.

We say that $\{S_i \to S\}$ is an {\it \'etale covering} if each $S_i \to S$ is \'etale and $\coprod S_i \to S$ is surjective.  To simplify the notation, we set $S_{ij} := S_i \times_S S_j$ and $S_{ijk} := S_i \times_S S_j \times_S S_k$.

\begin{definition}  A category $\cX$ fibered in groupoids over a scheme $S$ is a {\it stack over $S$} if 
 for any \'etale covering $\{T_i \to  T\}$ of an $S$-scheme $T$, we have:
\begin{enumerate}
\item (``morphisms glue") For objects $a,b$ in $\cX$ over $T$ and morphisms \linebreak $\phi_i\co a|_{T_i} \to b$ over $T_i \to T$ such that $\phi_i|_{T_{ij}} = \phi_j|_{T_{ij}}$, then there exists a unique morphism $\phi\co a \to b$ over the identity with $\phi|_{T_i} = \phi_i$.  Pictorially, we are requiring that a commutative diagram
$$
\begin{array}{c}
\xymatrix{
				& a|_{T_i} \ar[rd] \ar[rrd]^{\phi_i} \\
a|_{T_{ij}}\ar[ur] \ar[dr]	&			& a \ar@{-->}[r]^{\phi}	& b\\
				& a|_{T_j} \ar[ru] \ar[rru]_{\phi_j}
} \end{array}
 \text{ over }
 \begin{array}{c} \xymatrix{
				& T_i \ar[rd] \\
T_{ij} \ar[ur] \ar[dr]	&			& T\\
				& T_j \ar[ru]
} \end{array}
$$
can be completed in a unique way by an arrow $\phi\co a \to b$.
\item (``objects glue") For objects $a_i$ over $T_i$ with isomorphisms $\alpha_{ij}\co a_i|_{T_{ij}} \to a_j|_{T_{ij}}$ over $\id_{T_{ij}}$ satisfying the cocycle condition $\alpha_{ij} \circ \alpha_{jk} = \alpha_{ik}$ on $T_{ijk}$, then there exist a unique object $a$ over $T$ and isomorphisms $\phi_i\co a|_{T_i} \to a_i$ over $\id_{T_i}$ such that $\alpha_{ij} \circ \phi_i
|_{T_{ij}} = \phi_j|_{T_{ij}}$.  Pictorially, we are requiring that diagrams
$$
\begin{array}{c}
\xymatrix{
				& a_i \ar@{-->}[rd] \\
a_i|_{T_{ij}} \xrightarrow{\alpha_{ij}} a_j|_{T_{ij}}\ar[ur] \ar[dr]	&			& a \\
				& a_j \ar@{-->}[ru]
} \end{array}
 \text{ over }
 \begin{array}{c} \xymatrix{
				& T_i \ar[rd] \\
T_{ij} \ar[ur] \ar[dr]	&			& T\\
				& T_j \ar[ru]
} \end{array}
$$
where the $\alpha_{ij}$ satisfy the cocycle condition can be filled in with an object $a$ over $T$.
\end{enumerate}
\end{definition}

\begin{remark}
These gluing conditions are extremely natural and should already be familiar to you.  For instance, rather than the \'etale topology, consider the Zariski topology (i.e. only consider Zariski covers $S = \bigcup S_i$) and consider the category fibered in groupoids $\cX$ over $\Spec \ZZ$ parametrizing pairs $(T, \cF)$,  where $T$ is a scheme and $\cF$ is a quasi-coherent sheaf on $T$. A morphism $(T', \cF') \to (T, \cF)$ in $\cX$ is the data of a morphism $f \co T' \to T$ together with an isomorphism $\cF \iso f^* \cF'$.  Then the fact that $\cX$ is a stack (in the Zariski topology) is the basic fact that quasi-coherent sheaves and their morphisms can be uniquely glued.  In this case, $\cX$ is also a stack in the \'etale (or even fppf) topology.
\end{remark}

\begin{definition} Let $S$ be a scheme.
A stack $\cX$ over $S$ is {\it algebraic} if 
\begin{enumerate}
	\item the diagonal $\Delta \co \cX \to \cX \times_S \cX$ is representable, and
	\item there exist a scheme $U$ and a representable, smooth and surjective morphism $U \to \cX$.
\end{enumerate}
\end{definition}

\begin{remark} \label{R:isom}
The representability condition  in (1) means that for every scheme $T$ and pair of morphisms $\xi, \eta \co T \to \cX$, the fiber product $\cX \times_{\Delta, \cX \times_S \cX, (\xi, \eta)} T$ is an algebraic space.  This in turn translates to the condition that the functor
$$\Isom_T(\xi, \eta) \co \Sch/T \to \Sets, \qquad (T' \to T) \mapsto \Mor_{\cX(T')}(\xi|_{T'}, \eta|_{T'})$$
is representable by an algebraic space, where $ \Mor_{\cX(T')}(\xi|_{T'}, \eta|_{T'})$ denotes the set of morphisms (which are necessarily isomorphisms) in the fiber category $\cX(T')$ of pullbacks of $\xi$ and $\eta$ to $T'$.

In (2), the condition that $U \to \cX$ is representable, smooth and surjective means that for any morphism $T \to \cX$ where $T$ is a scheme, the fiber product $T \times_{\cX} U$ is an algebraic space and $T \times_{\cX} U \to T$ is smooth and surjective.  (In fact, condition (2) can be shown to imply condition (1).)

\end{remark}

\subsection{Artin's axioms} \label{S:artins-axioms}

A spectacular application of Artin's algebraization theorem is Artin's local criterion for algebraicity of stacks.  This is a foundational result in the theory of algebraic stacks and was proved by Artin 
in the very same paper \cite{artin-versal} where he introduced algebraic stacks.

We first state a conceptual version of Artin's axioms that can be proved easily using only Artin algebraization.

\begin{theorem}(Artin's axioms---first version) \label{T:artins-axioms1}
Let $\cX$ be a stack over $k$.  Then $\cX$ is an algebraic stack locally of finite type over $k$ if and only if the following conditions hold:
\begin{enumerate}
\item[(0)] (Limit preserving) The stack $X$ is limit preserving.
\item (Representability of the diagonal) The diagonal $\cX \to \cX \times \cX$ is representable.
\item (Existence of formal deformations)  For every $x \co \Spec k \to \cX$, there exist a complete local noetherian $k$-algebra $(R, \fm)$ and a compatible family of morphism $\xi_n \co \Spec R / \fm^{n+1} \to \cX$ with $x = \xi_0$ such that $\{\xi_n\}$ is formally versal (as defined in Remark \ref{R:formal-versality}).
\item (Effectivity) For every complete local noetherian $k$-algebra $(R, \fm)$, the natural functor
$$\cX(\Spec R) \to \ilim \cX(\Spec R/ \fm^n)$$
is an equivalence of categories.
\item (Openness of versality)  
For any morphism $\xi_U \co U \to \cX$ where $U$ is a scheme of finite type over $k$ and point $u \in U$ such that $\xi_U$ is formally versal at $u$ (i.e. the induced morphism $\Spec \hat{\oh}_{U,u} \to \cX$ is formally versal), then $\xi_U$ is formally versal in an open neighborhood of $u$.
\end{enumerate}
\end{theorem}

\begin{proof}
The ``only if" direction is fairly straightforward and left to the reader.  For the ``if" direction, 
we first remark that Condition (1) implies 
that any morphism $U \to \cX$ from a scheme $U$ is necessarily representable.

Let $x \co \Spec k \to \cX$ be a morphism.  We need to find a commutative diagram
$$\xymatrix{
\Spec k \ar[rd]^{x} \ar[r]^u						& U \ar[d] \\
	& \cX
}$$
where $U$ is a scheme and $U \to \cX$ is smooth.   Condition (2) and (3) guarantee that there exists a complete local noetherian $k$-algebra $(R,\fm)$ with $R/ \fm \cong k$ together with a commutative diagram
$$\xymatrix{
\Spec k \ar[r] \ar[rd]^x	& \Spec R \ar[d]^{\hat{\xi}} \\
				& \cX
}$$
such that $\hat{\xi}$ is formally versal.  
By Artin's algebraization theorem, we can find an affine scheme $U = \Spec A$ of finite type over $k$, a point $u \in U$, a morphism $\xi \co U \to \cX$, and an isomorphism $R \cong \hat{\oh}_{U,u}$ yielding a commutative diagram
$$\xymatrix{
\Spec k \ar@{^(->}[r]	\ar[rrd]_{x}		  &  \Spec R \ar[rd]^{\hat{\xi}} \ar[r] & U \ar[d]^{\xi}\\
													 & & \cX.
}$$
We know that $U \to \cX$ is formally versal at $u$ and condition (4) implies that it is formally versal in a neighborhood.  But this implies that $U \to \cX$ is smooth in a neighborhood of $u$.\footnote{This isn't a trivial implication.  The formal versality condition is only a condition on lifting local artinian $k$-algebras but a theorem of Grothendieck \cite[IV.17.14.2]{ega} implies that this is sufficient to guarantee smoothness.}
\end{proof}

\begin{remark} In practice, Condition (1) is often easy to verify directly.  Alternatively one could also apply the theorem to the diagonal $\cX \to \cX \times \cX$ (i.e. to the sheaves $\Isom_T(\xi, \eta)$ defined in Remark \ref{R:isom}).  
Condition (2) is often a consequence of Schlessinger and Rim's theorem on existence of formally versal deformations \cite{schlessinger}, \cite{rim}.  
Condition (3) is often a consequence of Grothendieck's existence theorem \cite[III.5.1.4]{ega}.  

In some simplified moduli problems, Condition (4) can be checked directly.  For instance, if for each point $x \co \Spec k \to \cX$, the formal deformation space $(R, \fm)$ (as in Condition (2)) is regular (or more generally normal, geometrically unibranch and free of embedded points), then Condition (4) is automatically satisfied; see \cite[Thm.~3.9] {artin-algebraization}.  In more general moduli problems, Condition (4) is often guaranteed as a consequence of a well-behaved deformation and obstruction theory.  This will be explained in the next section.
\end{remark}

\subsection{A more refined version of Artin's axioms}
We will now state a refinement of Theorem \ref{T:artins-axioms1} that is often easier to verify in practice.
In order to state the theorem succinctly, we need to introduce a bit of notation.  

Let $A$ be a finitely generated $k$-algebra  and let $M$ be a finite $A$-module.  Denote by $A[M]$ the ring $A \oplus M$ defined by $M^2 = 0$.  Let $\xi \co \Spec A \to  \cX$.  Denote by $\Def_\xi(M)$ the set of isomorphism classes of diagrams 
$$\xymatrix{
\Spec A \ar[r]^{\xi} \ar@{_(->}[d]		& \cX  \\
\Spec A[M] \ar@{-->}[ur]^{\eta}.
}$$
Let $\Aut_{\xi}(M)$ be the group of automorphisms of the trivial deformation \linebreak $\Spec A[M] \to \Spec A \to \cX$ which restrict to the identity automorphism of $\xi$.

We remark that $\Aut_{\xi}(M)$ naturally has the structure of an $A$-module.\footnote{Indeed, for $a \in A$, the $A$-algebra homomorphism $A[M] \to A[M], a_0+m_0 \mapsto a_0+am_0$ induces a morphism $f_a \co \Spec A[M] \to \Spec A[M]$ over $\Spec A$ which in turn induces a group homomorphism $\Aut_{\xi}(M) \to \Aut_{\xi}(M)$ obtained by pulling back automorphisms along $f_a$.}    Condition (1a) below implies that $\Def_{\xi}(M)$ is naturally an $A$-module.\footnote{This may seem surprising since in condition (1a), the $k$-algebras $A,A',B$ are only assumed to be artinian.  However, as shown in \cite{hall-rydh-versal}, this is strong enough to imply that for any finitely generated $k$-algebra $A$ and finite $A$-module $M$, the natural map
$$\cX(\Spec(A[M \oplus M])) \to \cX(\Spec A[M]) \times_{\cX(\Spec A)} \cX(\Spec A[M])$$
is an equivalence of categories.  Therefore, addition $M \oplus M \to M$ induces an $A$-algebra homomorphism $A[M \oplus M] \to A[M]$ and thus a functor $\cX(\Spec A[M]) \times_{\cX(\Spec A)} \cX(\Spec A[M]) \cong \cX(\Spec A[M \oplus M]) \to \cX(\Spec A[M])$ which defines addition on $\Def_{\xi}(M)$.  Multiplication by an element $a \in A$ induces an $A$-algebra homomorphism $A[M] \to A[M]$ and therefore a functor $\cX(\Spec A[M]) \to \cX(\Spec A[M])$ which defines multiplication by $a \in A$.}

Note that $\Aut_{\xi}(k)$ is the group of infinitesimal automorphism of $\xi$, and $\Def_{\xi}(k)$ is the set of isomorphism classes of deformations over the dual numbers $\Spec k[\epsilon]$  which can be  thought of as the Zariski tangent space.

\begin{theorem}[Artin's axioms---refined version]  \label{T:artin-axioms2}
Let $\cX$ be a stack over $k$.  Then $\cX$ is an algebraic stack locally of finite type over $k$ if and only if the following conditions hold:
\begin{enumerate}
\item[(0)] (Limit preserving) The stack $\cX$ is limit preserving.
\item (Existence of formal deformations)  
\begin{enumerate}
\item (Homogeneity) For every diagram
$$\xymatrix{
A \ar[r] \ar[d]		& A' \ar[d]\\
B \ar[r]			& A' \times_A B
}$$
of finitely generated local artinian $k$-algebras where $A \to A'$ is surjective, the natural functor
$$\cX(\Spec(A' \times_A B)) \to \cX(\Spec A') \times_{\cX(\Spec A)} \cX(\Spec B)$$
is an equivalence of categories.  
\item (Finiteness of tangent spaces)  For every object $\xi \co \Spec k \to \cX$, 
$\Aut_{\xi}(k)$ and $\Def_{\xi}(k)$ are finite dimensional $k$-vector spaces.
\end{enumerate}
\item (Effectivity) For every complete local noetherian $k$-algebra $(R, \fm)$, the natural functor
$$\cX(\Spec R) \to \ilim \cX(\Spec R/ \fm^n)$$
is an equivalence of categories.
\item (Openness of versality)  
\begin{enumerate}
\item (Coherent deformation theory)  For every finitely generated $k$-algebra $A$, finite $A$-module $M$ and object $\xi \co \Spec A \to \cX$, the $A$-modules $\Aut_{\xi}(M)$ and $\Def_{\xi}(M)$ are finite. Thus there are $A$-linear functors
$$\begin{aligned}
\Aut_{\xi} \co \FiniteA & \to  \FiniteA  \\
\Def_{\xi} \co \FiniteA & \to  \FiniteA .
\end{aligned}$$
\item (Existence of an obstruction theory) For every finitely generated $k$-algebra $A$ and object $\xi \co \Spec A \to \cX$, there exists an $A$-linear functor
$$\Ob_{\xi} \co \FiniteA  \to  \FiniteA. $$
Moreover, for each surjection $A' \to A$ with squarezero kernel $I$, there exists an element $\o_{\xi}(A') \in \Ob_{\xi}(I)$ such that there is an extension
$$\xymatrix{
\Spec A \ar[r]^{\xi} \ar@{_(->}[d]		& \cX  \\
\Spec A' \ar@{-->}[ur]
}$$
if and only if $\o_{\xi}(A') = 0$.
\item (Constructibility)  For every finitely generated $k$-algebra $A$, finite $A$-module $M$ supported on $A_{\red}$ and object $\xi \co \Spec A \to \cX$, there exists a dense open subset $U \subset \Spec A$ such that for every $k$-point $u \in U$, the natural maps
$$\begin{aligned}
\Aut_{\xi}(M) \otimes_A k(u) & \to \Aut_{\xi|_u}(M \otimes_A k(u)) \\
\Def_{\xi}(M) \otimes_A k(u) & \to \Def_{\xi|_u}(M \otimes_A k(u)) \\
\Ob_{\xi}(M) \otimes_A k(u) & \to \Ob_{\xi}(M \otimes_A k(u)) \\
\end{aligned}$$
are isomorphisms.
\end{enumerate}
\end{enumerate}
\end{theorem}

\begin{remark}
An analogous statement is true after replacing $\Spec k$ with an arbitrary excellent scheme after suitably modifying Conditions 1--3.  
\end{remark}

Artin proved a version of the above theorem in \cite{artin-versal}.
We will not give a complete proof of this statement here and restrict ourselves to only making a few comments.  First, one reduces to the case that $\cX \to \cX \times \cX$ is representable by bootstrapping the below argument using automorphisms and deformations (rather than deformations and obstructions) to conclude that the isomorphism sheaves (as defined in Remark \ref{R:isom}) are representable.

Conditions 1(a)--(b) above allow us to apply Schlessinger and Rim's theorem \cite{schlessinger}, \cite{rim} which guarantees the existence of formal deformation, i.e. Condition 1 in Theorem \ref{T:artins-axioms1} holds.  This condition is often fairly easy to check in practice for moduli problems.  

In practice, Condition 2 is often established as a direct consequence of Grothendieck's Existence Theorem.

Conditions 3(a)--(c) can be shown to imply that formal versality is an open condition, i.e. Condition 3 in Theorem  \ref{T:artins-axioms1} holds.  To this end, it is necessary to show that if $\xi \co U \to \cX$ is a morphism which is formally versal at $u \in U$, then for points $v \in U$ in a sufficiently small neighborhood of $u$, any commutative diagram of the form 
\begin{equation} \label{D:lifting}
\begin{split}
\xymatrix{
\Spec B \ar[r]  \ar@{^(->}[d]	&  U \ar[d]^{\xi} \\
\Spec B' \ar[r] \ar@{-->}[ur]			& \cX,
} \end{split}
\end{equation}
where $B' \to B$ is a surjection of artinian $k$-algebras and $\Spec B \to U$ maps to $v$, has an extension.  This is a deformation problem.  Vaguely speaking, conditions 3(a)--(c) imply that the deformation theory of $U \to \cX$ is controlled in some sense by a coherent module, and formal versality at $u$ implies that this coherent module vanishes at $u$ and thus in an open neighborhood.   

Condition 3 certainly takes the most time and space to formulate. Moreover the most difficult part of the proof of Theorem \ref{T:artin-axioms2} is to show that Condition 3 implies openness of versality.  Nevertheless it is quite easy to establish Condition 3 in practice.  For many moduli problems, one shows that if $M$ is a finite $A$-module, then the $A$-modules $\Aut_{\xi}(M)$, $\Def_{\xi}(M)$, and $\Ob_{\xi}(M)$ are naturally identified with certain cohomology modules in which case Condition 3(c) can be seen to follow from cohomology and base change.  For example, if $\cX$ is the moduli space of smooth curves as in Example \ref{E:curves} and $\xi$ corresponds to a curve $C \to \Spec A$, then $\Aut_{\xi}(M) = \H^0(C, T_{C} \otimes M)$, $\Def_{\xi}(M) = \H^1(C, T_{C} \otimes M)$ and $\Ob_{\xi}(M) = \H^2(C, T_{C} \otimes M) = 0$.  Here the constructibility condition in 3(c) follows directly from cohomology and base change applied at the generic points of $\Spec A$.

We recommend \cite{hall-versal} for a conceptual proof of Artin's criterion.
There is some flexibility in how one precisely formulates Conditions 1 and 3.  We recommend \cite{hall-rydh-versal} for a technical account of various formulations of Artin's axioms and in particular for a complete proof of the formulation given in Theorem \ref{T:artin-axioms2}.

\newpage

\section*{Lecture 3: The geometry of quotient stacks}
\refstepcounter{section}

In this lecture, we discuss a particularly important example of an algebraic stack, namely the {\it quotient stack} $[X/G]$, which arises from the action of an algebraic group $G$ on a scheme $X$. 
We will emphasize that the geometry of a quotient stack $[X/G]$ is nothing other than the $G$-equivariant geometry of $X$.  In this lecture, we discuss when a general algebraic stack is a quotient stack and then turn to the main theorem of these lectures which gives conditions for an algebraic stack to be \'etale locally a quotient stack.

\subsection{Quotient stacks} \label{S:quotient-stacks}

\begin{definition}  Let $G$ be a smooth affine group scheme over $k$ acting on a scheme $X$ over $k$.  We define the {\it quotient stack of $X$ by $G$}, denoted by $[X/G]$, to be the category of objects $(T, P \xrightarrow{\pi} T, P \xrightarrow{g} X)$ where $T$ is a scheme, $P \to T$ is a principal $G$-bundle and $P \to X$ is a $G$-equivariant morphism.  A morphism $(T, P \xrightarrow{\pi} T, P \xrightarrow{g} X) \to (T', P' \xrightarrow{\pi'} T', P' \xrightarrow{g'} X)$ in this category consists of a commutative diagram 
$$\xymatrix{
P \ar[d]^{\pi} \ar[r]  \ar@/^4ex/[rr]^{g}	& P' \ar[d]^{\pi'} 	\ar[r]^{g'}		& X \\
T \ar[r] \ar@{}[ur]|\square			& T'
}$$
where the square is cartesian.
\end{definition}

\begin{remark} With the above hypotheses, $[X/G]$ is an algebraic stack.  Here Axiom (2) of the definition of an algebraic stack is satisfied by the projection $X \to [X/G]$.  The morphism $X \to [X/G]$ corresponds via Yoneda's lemma to the object of $[X/G]$ over $X$ defined by the trivial $G$-bundle $G \times X \to X$ together with the $G$-equivariant morphism $\sigma \co G \times X \to X$ (corresponding to the action of $G$ on $X$).  The morphism $X \to [X/G]$ is a $G$-torsor.
\end{remark}

\begin{remark}
In fact, in characteristic $p$, if $G$ is not smooth, it can still be shown that $[X/G]$ is an algebraic stack.
\end{remark}

\begin{definition}
We say that an algebraic stack $\cX$ is a {\it global quotient stack} if
$\cX \cong [U / \GL_n]$ where $U$ is an algebraic space with an action of $\GL_n$.
\end{definition}

\begin{remark}   If $G$ is an affine group scheme of finite type over $k$ and $U$ is an algebraic space over $k$ with an action of $G$, then $[U/G]$ is a global quotient stack.  To see this, choose a faithful representation $G \subset \GL_n$.  Then $[U/G] \cong [ (U \times^{G} \GL_n) / \GL_n]$ where $U \times^{G} \GL_n = (U \times \GL_n) / G$ (and here $G$ acts diagonally on $U \times \GL_n$).
\end{remark}

Quotient stacks provide very important examples of algebraic stacks as their geometry is particularly well understood.  In fact, the geometry of a quotient stack $\cX = [X/G]$ is nothing other than the $G$-equivariant geometry of $X$.  To justify this philosophy, we provide a dictionary between geometric concepts of $[X/G]$ and $G$-equivariant geometric concepts of $X$.

\medskip

\begin{center}
{\renewcommand{\arraystretch}{1.5}
\renewcommand{\tabcolsep}{0.2cm}
\begin{tabular}{p{7cm} | p{7cm}} 
\hline
\begin{center} Geometry of $\cX = [X/G]$ \end{center} & \begin{center} $G$-equivariant geometry of $X$ \end{center}\\
\hline
a point $x \co \Spec k \to \cX$	& a $G$-orbit $G x \subset X$ \\
\hline
the automorphism group $\Aut_{\cX(k)}(x)$	& the stabilizer $G_x$ \\
\hline
a function $f \in \Gamma(\cX, \oh_{\cX})$	& a $G$-invariant function $f \in \Gamma(X, \oh_X)^G$\\
\hline
a morphism $\cX \to Y$ where $Y$ is a scheme (or an algebraic space)	& a $G$-invariant morphism $X \to Y$ \\
\hline
a line bundle on $\cX$ 	& a line bundle on $X$ with a $G$-action \\
\hline
a coherent $\oh_{\cX}$-module	& a coherent $\oh_X$-module with a $G$-action \\
\hline
properties of the diagonal $\cX \to \cX \times \cX$		& properties of the group action $G \times X \to X \times X$ \\
\hline
the tangent space $T_{\cX,x}$	& the normal space to the orbit $T_{X,x} / T_{G x, x}$ \\
\end{tabular}}
\end{center}

Above, we used the notion of the tangent space of a stack at a point.  Since this will be important later, let's define it precisely.

\begin{definition} \label{D:tangent-space}
If $\cX$ is an algebraic stack of finite type over $k$ and $x$ is a $k$-point, then the {\it tangent space of $\cX$ at $x$}, denoted by $T_{\cX, x}$, is defined as the set of isomorphism classes of extensions
$$\xymatrix{
\Spec k \ar@{^(->}[d] \ar[r]^x		& \cX \\
\Spec k[\epsilon]/\epsilon^2. \ar@{-->}[ru]
}$$
\end{definition}

\begin{remark}    
The equivalence $T_{\cX, x} \cong T_{X,x} / T_{Gx,x}$ only holds when the stabilizer $G_x$ is smooth.
\end{remark}

\begin{example}  Consider the action of the multiplicative group $\GG_m = \Spec k[t]_t$ on $\AA^n$ via multiplication.  Then $[\AA^n / \GG_m]$ is an algebraic stack.  The origin is the only closed $k$-point of $[\AA^n / \GG_m]$ as all other $k$-points have the origin in their closure.  As $[ (\AA^n \setminus 0 ) / \GG_m] = \PP^{n-1}$, there is an open substack $\PP^{n-1} \subset [\AA^n / \GG_m]$.
\end{example}

\begin{example}  Consider the action of $\GG_m$ on $\AA^2$ via $t \cdot (x,y) = (tx, t^{-1} y)$.  Then the closed $k$-points of $[\AA^2 / \GG_m]$ correspond to the origin together with the hyperbolas $\{xy=a\} \subset \AA^2$ (i.e. the $\GG_m$-orbit of $(a,1)$ for $a \neq 0$).  The two orbits $\GG_m (1,0)$ and $\GG_m (0,1)$ both have the origin in their closure.  Here the open substack $[(\AA^2 \setminus 0 ) / \GG_m] \subset [\AA^2 / \GG_m]$ is isomorphic to the non-separated affine line with a double origin.  The morphism $[\AA^2 / \GG_m] \to \AA^1$ given by $(x,y) \mapsto xy$ gives a bijective correspondence between closed $k$-points in $[\AA^2 / \GG_m]$ and $k$-points in $\AA^1$.  Note that the fiber over the origin consists of $3$ points corresponding to the orbits in the union of the $x$ and $y$-axes.
\end{example}

A central question is:

\begin{question}
When is an algebraic stack $\cX$ a global quotient?
\end{question}

This question is very difficult.  As we will see below, the question of whether $\cX$ is a global quotient is related to other global geometric properties of $\cX$.

\subsection{A summary of known results on quotient stacks}

We first make a basic observation.  In this discussion, we restrict ourselves to the case that $\cX$ is a quasi-separated algebraic stack of finite type over $k$.  If $\cX = [X/\GL_n]$ is a global quotient, then $X \to \cX$ is a $\GL_n$-torsor and one can construct $[X \times \AA^n / \GL_n]$ which is a vector bundle over $\cX$ with the property that  each stabilizer acts faithfully on the fiber.  Conversely, given a vector bundle $\cV \to \cX$ with this same property, then the frame bundle $\Frame(\cV)$ is an algebraic space and $\Frame(\cV) \to \cX$ is a $\GL_n$-torsor.  We conclude that
\medskip
\begin{center}
\begin{tabular} {p{4.3cm} c p{7cm}}
$\cX \cong [ \text{alg space} / \GL_n]$ (i.e. a global quotient) & $\iff$ & there exists a vector bundle $\cV$ on $\cX$ such that for every point $x$, the stabilizer $G_x$ acts faithfully on the fiber $\cV \tensor k(x)$.
\end{tabular}
\end{center}

In a similar spirit, a theorem due to Totaro \cite{totaro} (generalized by Gross \cite{gross} to the non-normal case) implies
\medskip
\begin{center}
\begin{tabular} {p{4.3cm} c p{7cm}}
$\cX \cong [ \text{quasi-affine} / \GL_n]$ & $\iff$ & $\cX$ satisfies the resolution property (i.e. every coherent $\oh_{\cX}$-module is surjected onto by some vector bundle) \\
\\
$\cX \cong [ \text{affine} / \GL_n]$ &  $\stackrel{\text{char(k) = 0}}{\iff}$
 & $\cX$ satisfies the resolution property, has affine diagonal, and $\H^i(\cX, \cF) = 0$ for every coherent $\oh_{\cX}$-module $\cF$ and $i > 0$.\\
\end{tabular}
\end{center}

The last case when $\cX \cong [\Spec A / \GL_n]$ provides quotient stacks of the simplest structure. In fact, it is useful to replace $\GL_n$ with any linearly reductive group scheme $G$.
Recall that an affine group scheme $G$ of finite type over $k$ is called {\it linearly reductive} if the functor
$$ G \mhyphen \Rep \to k\mhyphen\Vect, \qquad V \to V^G$$
is exact (or equivalently every $G$-representation is completely reducible).   
 If $G$ is a linearly reductive group (e.g. $G =\GL_n$ in characteristic 0) acting on an affine scheme $X = \Spec A$, then there is an affine GIT quotient
$$[\Spec A / G ] \to \Spec A \gitq G := \Spec A^{G}$$
whose geometry is very well understood.

We now mention two nice results from \cite{ehkv}.

\begin{proposition} \cite[Thm.~2.18]{ehkv}
If $\cX$ is a smooth and separated Deligne-Mumford stack of finite type over $k$ with generically trivial stabilizer, then $\cX$ is a global quotient.
\end{proposition}

\begin{proposition} \cite[Thm.~3.6]{ehkv}
Let $X$ be a noetherian scheme and let $\cX \to X$ be a $\GG_m$-gerbe corresponding to $\alpha \in \H^2(X, \GG_m)$.  Then $\cX$ is a global quotient if and only if $\alpha$ is in the image of the Brauer map $\Br(X) \to \H^2(X, \GG_m)$.
\end{proposition}

\begin{remark} This latter proposition can be used to construct algebraic stacks that are not global quotient stacks, including a non-separated Deligne-Mumford stack and a normal (but non-smooth) algebraic stack with affine diagonal.  See \cite[Examples 2.21 and 3.12]{ehkv}.
\end{remark}

The following two questions are completely open.

\begin{question}  Is every separated Deligne-Mumford stack a global quotient stack?
\end{question}

\begin{question}  Does every smooth algebraic stack with affine diagonal satisfy the resolution property?
\end{question}

The question of whether a given algebraic stack is a global quotient stack appears very difficult and is related to both global geometric properties (such as existence of vector bundles) as well as arithmetic questions (such as the surjectivity of the Brauer map).  Below we will attempt to address the simpler question:  when are algebraic stack \'etale locally quotient stacks?

\subsection{The local quotient structure of algebraic stacks}

Recall that $k$ denotes an algebraically closed field of any characteristic.  Also recall that we denote by $T_{\cX, x}$ the tangent space of a stack $\cX$ at $x$; see Definition \ref{D:tangent-space}.

\begin{theorem} \label{T:smooth}
Let $\cX$ be a quasi-separated algebraic stack, locally 
of finite type over $k$, with affine stabilizers.  Let $x \in \cX$ be a smooth $k$-point with smooth and linearly reductive stabilizer
  group $G_x$. Then there exist an affine and \'etale morphism $(U,u) \to (T_{\cX,x} \gitq G_x,0)$,  and a cartesian diagram
 $$\xymatrix{
 \bigl([T_{\cX,x}/G_x],0\bigr) \ar[d] & \bigl([\Spec A / G_x],w\bigr)\ar[r]^-f \ar[d] \ar[l]	& (\cX,x) \\
 (T_{\cX,x} \gitq G_x,0) & (U,u) \ar[l]	\ar@{}[ul]|\square				& 
}$$
such that  $f$ is \'etale and induces an isomorphism
    of stabilizer groups at $w$. 
\end{theorem}

\begin{remark}  This theorem was established in \cite[Thm.~1.1]{ahr}.  The theorem is true even if the stabilizer $G_x$ is not smooth if one replaces the tangent space $T_{\cX,x}$ with the normal space $N_x = (\cI/\cI^2)^\vee$, where $\cI \subset \oh_{\cU}$
denotes the sheaf of ideals defining $x$ in an open substack $\cU\subset \cX$ where $x$ is a closed point.  In the case that the stabilizer is smooth, $T_{\cX, x} \cong N_x$.
\end{remark}

This theorem implies that $\cX$ and $[T_{\cX,x} /G_x]$ have a common \'etale neighborhood of the form $[\Spec A / G_x]$. 

In the case that $x$ is not necessarily a smooth point of $\cX$, one can prove a similar structure theorem:

\begin{theorem}\cite[Thm.~1.2]{ahr}\label{T:general}
Let $\cX$ be a quasi-separated algebraic stack, locally 
of finite type over an algebraically closed field $k$, with affine stabilizers.  Let $x \in \cX$ be a $k$-point with a linearly reductive stabilizer.  Then there exist an affine scheme $\Spec A$ with an action of $G_x$, a $k$-point $w \in \Spec A$ fixed by $G_x$, and an \'etale morphism
$$f\colon \bigl([\Spec A/G_x],w\bigr) \to (\cX,x)$$
such that $f$ induces an isomorphism
    of stabilizer groups at $w$. 
\end{theorem}

These theorems  
justify the philosophy that quotient stacks of 
the form $[\Spec A / G]$, where $G$ is a linearly reductive group, are the building blocks
 of algebraic stacks near points with linearly reductive stabilizers in the same way that affine schemes are the building blocks of schemes (and algebraic spaces).
 
 These theorems were known in the following special cases:
\begin{enumerate}
	\item $\cX$ is a Deligne-Mumford stack \cite{abramovich-vistoli}.
	\item More generally, $\cX$ has quasi-finite inertia (i.e. all stabilizers groups are finite but not necessarily reduced).  This follows from \cite[Thm.~3.2]{tame} and \cite[\S 4]{keel-mori}. 
	\item $\cX \cong [X/ G]$ where $G$ is a linearly reductive algebraic group acting on an affine scheme $X$ and $G_x$ is smooth and linearly reductive.  This case is often referred to as Luna's \'etale slice theorem \cite[p.~97]{luna}.  We will discuss the relation between the theorems above and Luna's \'etale slice theorem in Section \ref{S:luna}.  
	We do emphasize though that there is still content in the theorems even in the case that $\cX \cong [\Spec A / G]$ as the theorems provide \'etale neighborhoods which are quotient stacks of affine schemes by the {\it stabilizer}.
	\item $\cX \cong [X/G]$ where $G$ is a smooth affine group scheme and $X$ is a normal scheme \cite[\S2.2]{alper-kresch}.
	\item $\cX = \fM_{g,n}^{\ss}$ is the 
moduli stack of semistable curves. This is the central result of \cite{alper-kresch}, where it is also shown that $f$ can taken to be representable.
\end{enumerate}

We mention here counterexamples to  Theorems \ref{T:smooth} and \ref{T:general} if either the linearly reductive hypothesis or the condition of affine stabilizers is weakened.

\begin{example} \label{ex1}
Some reductivity assumption of the stabilizer $G_x$ is necessary in Theorem \ref{T:general}.  For instance,  consider the group scheme $G= \Spec k[x,y]_{xy+1} \to \AA^1 = \Spec k[x]$ (with multiplication defined by $y \mapsto xyy' + y + y'$), where the generic fiber is $\GG_m$ but the fiber over the origin is $\GG_a$.  Let $\cX = BG$ and $x \in \cX$ be the point corresponding to the origin.  There does not exist an \'etale morphism $([W/ \GG_a], w) \to (\cX, x)$, where $W$ is an algebraic space over $k$ with an action of $\GG_a$.
\end{example}

\begin{example} \label{ex2}
It is essential to require that the stabilizer groups are affine in a neighborhood of $x \in \cX$.  For instance, let $X$ be a smooth curve and $\cE \to X$ be a group scheme whose generic fiber is a smooth elliptic curve but the fiber over a point $x \in X$ is isomorphic to $\GG_m$.  Let $\cX = B\cE$.  There is no \'etale morphism $([W/ \GG_m], w) \to (\cX, x)$, where $W$ is an affine $k$-scheme with an action of $\GG_m$.
\end{example}

\subsection{Ingredients in the proof of Theorem \ref{T:smooth}}

There are four main techniques employed in the proof of Theorem \ref{T:smooth}:
\begin{enumerate}
\item deformation theory,
\item coherent completeness,
\item Tannakian formalism, and
\item Artin approximation.
\end{enumerate}

Deformation theory produces an isomorphism between the $n$th infinitesimal neighborhood $\cT_n$ of $0$ in $\cT = [T_{\cX,x} / G_x]$ and the $n$th infinitesimal neighborhood $\cX_n$ of $x$ in 
$\cX$. It is not at all obvious, however, that the system of closed morphisms $\{f_n \co \cT_n \to \cX\}$ algebraizes. We establish algebraization in two steps. 

The first step is effectivization. To accomplish this, we prove a result similar in spirit to Grothendieck's existence theorem \cite[III.5.1.4]{ega}, which we refer to as \emph{coherent completeness}.  To motivate the definition, recall that if $(A,\fm)$ is a complete local noetherian ring, then 
$\Coh(\Spec A) = \ilim_n \Coh( \Spec (A/\fm^{n+1}))$.  Here, if $X$ is a noetherian scheme, then $\Coh(X)$ denotes the category of coherent $\oh_{X}$-modules.

\medskip

\begin{blackbox}(Coherent completeness) \label{BB:complete}
Let $G$ be a linearly reductive affine group scheme over an algebraically closed field $k$. Let $\Spec A$ be a noetherian affine scheme with an action of~$G$, and let $x \in \Spec A$ be a $k$-point fixed by $G$.   Suppose that $A^{G}$ is a complete local ring. Let $\cX = [\Spec A / G]$ and let $\cX_n$ be the $n$th infinitesimal neighborhood of $x$.  Then the natural functor
\begin{equation} \label{eqn-coh}
\Coh(\cX)  \to  \ilim_n \Coh\bigl(\cX_n\bigr)
\end{equation}
is an equivalence of categories.
\end{blackbox}

\begin{remark}
This was proven in \cite[Thm.~3]{ahr}.  The proof is not tremendously difficult but does require some care.
\end{remark}

\begin{remark}  Let $Y = \Spec A^G$ and let $Y_n = \Spec A^G / (\fm \cap A^G)^{n+1}$ be the $n$th nilpotent thickening of the image of $x$ under $\Spec A \to \Spec A^G$. The above theorem implies that
\begin{equation} \label{E:formal-gaga}
\Coh(\cX) \to \ilim_n \Coh(\cX \times_Y Y_n)
\end{equation}
is an equivalence, which had been established in \cite{geraschenko-brown}.  We emphasize that the above theorem is significantly stronger in that it involves families of coherent sheaves defined only on the nilpotent thickenings of the closed point $x \in \cX$ rather the fiber $\cX \times_Y Y_0$.  
\end{remark}

\begin{example}
Consider $\cX = [\AA^N / \GG_m]$.  Then the theorem states that $\GG_m$-equivariant sheaves on $\AA^N$ are equivalent to compatible families of $\GG_m$-equivariant sheaves on $\Spec k[x_1, \ldots, x_N] / (x_1, \ldots, x_N)^{n+1}$.  Meanwhile, the equivalence \eqref{E:formal-gaga} is trivial as $Y_n = Y = \Spec k$.
\end{example}

The other key ingredient in the proof of Theorem \ref{T:smooth} is Tannakian formalism.  

\medskip

\begin{blackbox}(Tannakian formalism) \label{BB:tannakian}
Let $\cX$ be an excellent stack and $\cY$ be a noetherian algebraic stack with affine stabilizers.  Then the natural functor
$$\Hom(\cX, \cY) \to \Hom_{r\otimes, \simeq}\bigl(\Coh(\cY), \Coh(\cX)\bigr)$$
 is an equivalence of categories, where $\Hom_{r\otimes, \simeq}(\Coh(\cY), \Coh(\cX))$ denotes the category whose objects are right exact monoidal functors $\Coh(\cY) \to \Coh(\cX)$ and morphisms are natural isomorphisms of functors. 
\end{blackbox}

\begin{remark}
In the above generality, this result was established in \cite[Thm.~1.1]{hr-tannaka}.  Other versions had been established in \cite{lurie} and \cite{bhatt-halpern-leistner}.
\end{remark}

This proves that morphisms between algebraic stacks $\cY \to \cX$ are equivalent to symmetric monoidal functors $\Coh(\cX) \to \Coh(\cY)$. Therefore, to prove Theorem \ref{T:smooth}, we can  combine Black Box \ref{BB:complete} with Black Box \ref{BB:tannakian} and the above deformation-theoretic observations to show that the morphisms $\{f_n \colon \cT_n \to \cX\}$ effectivize to $\hat{f} \colon \hat{\cT} \to \cX$, where $\hat{\cT} = T_{\cX,x} \times_{T_{\cX,x} \gitq G_x} \Spec \hat{\oh}_{T_{\cX,x} \gitq G_x,0}$. 
The morphism $\hat{f}$ is then algebraized using Artin approximation over the GIT quotient $T_{\cX,x} \gitq G_x$.
We will give the details of this argument in the next lecture.

\newpage
  
\section*{Lecture 4: A Luna \'etale slice theorem for algebraic stacks and applications}
\refstepcounter{section}

In this lecture, we will give proofs of Theorems \ref{T:smooth} and \ref{T:general}.  Recall that these theorems assert roughly that any algebraic stack with affine stabilizers is \'etale locally a quotient stack in a neighborhood of a point with a linearly reductive stabilizer.  The proof of Theorem \ref{T:general} will rely on an equivariant version of Artin approximation (Theorem \ref{T:equivariant-algebraization}).  We will also give several applications of Theorems \ref{T:smooth} and \ref{T:general}.

Throughout this lecture, we will use the notation that if $\cX$ is an algebraic stack over a field $k$ and $x \in \cX$ is a closed $k$-point, then $\cX_n$ denotes the $n$th nilpotent thickening of the inclusion of the residual gerbe $BG_x \hookrightarrow \cX$; more precisely, if $BG_x \hookrightarrow \cX$ is defined by the sheaf of ideals $\cI$, then $\cX_n \hookrightarrow \cX$ is defined by $\cI^{n+1}$.
The point $x$ is not included in the notation $\cX_n$ but it should always be clear from the context.

\subsection{Proof of Theorem \ref{T:smooth}}

\begin{proof}[Proof of Theorem \ref{T:smooth}]
We may assume that $x \in \cX$ is a closed point.
Define the quotient stack $\cT= [T_{\cX,x}/G_x]$.  Since $G_x$ is linearly reductive and $x \in \cX$ is a smooth point, a simple deformation theory argument implies that there are isomorphisms $\cX_n \cong \cT_n$. 

Let $\cT  \to  T = T_{\cX,x} \gitq G_x$ be the morphism to the GIT quotient, and denote by $0 \in T$ the image of the origin.    The fiber product $\hat{\cT} := \Spec \hat{\oh}_{T,0} \times_T \cT$ is noetherian and a quotient stack of the form $[\Spec A / G]$ with $A^G$ a complete local ring.  Therefore, $\hat{\cT}$ satisfies the hypotheses of Black Box \ref{BB:complete}.  We have equivalences
\begin{align*}
\Hom(\hat{\cT}, \cX) 	& \simeq \Hom_{r\otimes, \simeq}\bigl( \Coh(\cX), \Coh(\hat{\cT})\bigr)  & & \text{(Tannakian formalism)}\\
				& \simeq \Hom_{r\otimes, \simeq}\bigl( \Coh(\cX), \ilim \Coh\bigl(\cT_n\bigr) \bigr)  & & \text{(coherent completeness)}\\
				& \simeq \ilim \Hom_{r\otimes, \simeq}\bigl( \Coh(\cX), \Coh\bigl(\cT_n\bigr) \bigr)  & & \\
				& \simeq \ilim \Hom\bigl(\cT_n, \cX\bigr) & & \text{(Tannakian formalism)}.
\end{align*}
Thus the morphisms $\cT_n \cong \cX_n \hookrightarrow \cX$ extend to a morphism $\hat{\cT} \to \cX$ filling in the diagram
\vspace{.2cm}
$$
\xymatrix{
\cX_n \cong \cT_n \ar[r] \ar@/^1.6pc/[rrr]	& \hat{\cT} \ar[r] \ar[d]	\ar@/^1pc/@{-->}[rr]					& \cT \ar[d] & \cX\\
& \Spec \hat{\oh}_{T,0} \ar[r]	\ar@{}[ur]|\square				& T.
}
$$

The functor 
$$F \co \Sch/T  \to  \Sets, \qquad (T' \to T) \mapsto \big\{ T' \times_T \cT \to \cX\big\} / \sim$$
is easily checked to be limit preserving.  The morphism $\hat{\cT} \to \cX$ yields an element of $F$ over $\Spec \hat{\oh}_{T,0}$.  By Artin approximation, there exist an \'etale morphism $(U,u) \to (T,0)$ where $U$ is an affine scheme and a morphism $(U \times_T \cT, (u,0) ) \to  (\cX,x)$ agreeing with $(\hat{\cT},0) \to (\cX,x)$ to first order.  Since $U \times_T \cT$ is smooth at $(u,0)$ and $\cX$ is smooth at $x$, and since $U \times_T \cT \to \cX$ induces an isomorphism of tangent spaces at $(u,0)$, the morphism $U \times_T \cT \to \cX$ is \'etale at $(u,0)$.  After shrinking $U$ suitably, the theorem is established.
\end{proof}

\begin{remark}
The smoothness hypothesis of $x \in \cX$ is used above to establish the isomorphisms $\cT_n \cong \cX_n$ as well as the \'etaleness of $U \times_T \cT \to \cX$.  More critically, though, it implies that $\ilim \Gamma(\cX_n, \oh_{\cX_n})$ is  the completion of a finitely generated $k$-algebra since this inverse limit is identified with $ \hat{\oh}_{T,0}$.  If $x \in \cX$ is not smooth, there does not appear to be a direct way to establish that $\ilim \Gamma(\cX_n, \oh_{\cX_n})$ (which can be identified with the $G_x$-invariants of a miniversal deformation space of $G_x$) is the completion of a finitely generated $k$-algebra.    Recall that Artin algebraization was a direct consequence of Artin approximation in the case that the complete local ring was known to be the completion of a finitely generated algebra (see Remark \ref{R:proof-algebraization-special-case}).  In a similar manner, one could deduce Theorem \ref{T:general} if we did know that $\ilim \Gamma(\cX_n, \oh_{\cX_n})$ was the completion of a finitely generated algebra.  In order to circumvent this problem, we will establish an equivariant version of Artin algebraization.
\end{remark}

\subsection{Equivariant Artin algebraization}

\begin{theorem}[Equivariant Artin algebraization]\label{T:equivariant-algebraization}
Let $H$ be a linearly reductive affine group scheme over $k$.
Let $\cX$ be a limit preserving category fibered in groupoids over $k$. 
Let $\cZ = [\Spec A / H]$ be a noetherian algebraic stack over $k$. Suppose that $A^{H}$ is a complete local $k$-algebra.  Let $z \in \cZ$ be the unique closed point.
 Let $\hat{\xi} \colon \cZ \to \cX$ be a morphism
that is formally versal at $z\in \cZ$.\footnote{This is a stacky extension of the notion of formal versality introduced in Definition \ref{D:versal}.  Namely, it means 
 that  for every 
commutative diagram of solid arrows
\[
\xymatrix{
\cZ_0 \ar@{(->}[r]
  & \cT\ar[r]\ar@{(->}[d]
  & \cZ\ar[d]^{\hat{\xi}} \\
& \cT'\ar[r]\ar@{-->}[ur] & \cX
}
\]
where $\cT \hookrightarrow \cT'$ is a closed immersion of local artinian stacks over $k$ (i.e. noetherian algebraic stacks over $k$ with a unique point), there is a lift $\cT' \to \cZ$ filling in the above diagram.
}  Then there exist 
\begin{enumerate}
\item
  an algebraic stack $\cW = [\Spec B / H]$ of finite type over $k$ and a closed point $w \in \cW$;
\item
  a morphism
$\xi \co \cW \to \cX$; 
\item an isomorphism $\cZ \cong \hat{\cW}_w$, where $\hat{\cW}_w$ is defined as the fiber product
$$\xymatrix{
\hat{\cW}_w \ar[r] \ar[d]	& \cW \ar[d]^{\phi} \\
\Spec \hat{\oh}_{W, \phi(w)} \ar[r] \ar@{}[ur]|\square	& W = \Spec B^H ; 
}$$
\item a compatible family of isomorphisms $\hat{\xi}|_{\cZ_n} \cong \xi|_{\cW_n}$ (under the identification $\cZ_n \cong \cW_n$) for $n \ge 0$.
  \end{enumerate}
\end{theorem}

\begin{remark} If one takes $H$ to be the trivial group, one recovers precisely the statement of Artin algebraization given in Theorem \ref{T:algebraization}.
\end{remark}

\begin{proof} This can be proved in a similar fashion to Theorem \ref{T:algebraization} by appealing to a stacky generalization of Theorem \ref{T:conrad-dejong}.  See \cite[App. A]{ahr}.
\end{proof}

\subsection{Proof of Theorem \ref{T:general}}

\begin{proof}[Proof of Theorem \ref{T:general}]
We may assume that $x \in \cX$ is a closed point.
Let $\cT= [T_{\cX,x}/G_x]$, $\cT  \to  T = T_{\cX,x} \gitq G_x$ be the morphism to the GIT quotient,  
and $\hat{\cT} := \Spec \hat{\oh}_{T,0} \times_T \cT$ where $0 \in T$ is the image of the origin.
Since $G_x$ is linearly reductive, a simple deformation theory argument implies that there are compatible closed immersions $\cX_n \hookrightarrow \cT_n$.  The ideals sheaves $\cI_n$ defining these closed immersions give a compatibly family $\{ \oh_{\cX_n} / \cI_n \}$ of coherent $\oh_{\cX_n}$-modules.  Since $\hat{\cT}$ satisfies the hypotheses of Black Box \ref{BB:complete}, there exists an ideal sheaf $\cI \subset \oh_{\hat{\cT}}$ such that the surjection $\oh_{\hat{\cT}} \to \oh_{\hat{\cT}} / \cI$ extends the surjections $\oh_{\cX_n} \to  \oh_{\cX_n} / \cI_n$. Therefore there exists a closed immersion $\cZ \hookrightarrow \hat{\cT}$ extending the given closed immersions $\cX_n \hookrightarrow \cT_n$.  This yields a commutative diagram
\vspace{.2cm}
$$
\xymatrix{
\cX_n \ar@{^(->}[d] \ar@{^(->}[r] 	\ar@/^1.6pc/[rrr]	& \cZ \ar@{^(->}[d]  \ar@{-->}[rr]^{\hat{\xi}}	& & \cX\\
\cT_n  \ar@{^(->}[r] 	& \hat{\cT} \ar[r] \ar[d]					& \cT \ar[d] & \\
& \Spec \hat{\oh}_{T,0} \ar[r]	\ar@{}[ur]|\square				& T.
}
$$
of solid arrows.  Since $\cZ$ also satisfies the hypotheses of Black Box \ref{BB:complete} and the nilpotent thickenings $\cZ_n$ are identified with $\cX_n$, the equivalences
\begin{align*}
\Hom(\cZ, \cX) 	& \simeq \Hom_{r\otimes, \simeq}\bigl( \Coh(\cX), \Coh(\cZ)\bigr)  & & \text{(Tannakian formalism)}\\
				& \simeq \Hom_{r\otimes, \simeq}\bigl( \Coh(\cX), \ilim \Coh\bigl(\cX_n\bigr) \bigr)  & & \text{(coherent completeness)}\\
				& \simeq \ilim \Hom_{r\otimes, \simeq}\bigl( \Coh(\cX), \Coh\bigl(\cX_n\bigr) \bigr)  & & \\
				& \simeq \ilim \Hom\bigl(\cX_n, \cX\bigr) & & \text{(Tannakian formalism)}.
\end{align*}
imply the existence of a morphism $\hat{\xi} \co  \cZ \to \cX$ filling in the above diagram.  One can check easily that $\hat{\xi} \co \cZ \to \cX$ is formally versal.  By applying equivariant Artin algebraization with (Theorem \ref{T:equivariant-algebraization}) with $H=G_x$, we obtain a morphism $\xi \co \cW = [\Spec B / G_x] \to \cX$ where $\cW$ is of finite type over $k$, a closed point $w \in \cW$ and an isomorphism
 $\cZ \to \hat{\cW}_w$ over $\cX$, where $\hat{\cW}_w$ is defined as in the statement of Theorem \ref{T:equivariant-algebraization}.  (The fact that $\cZ \to \hat{\cW}_w$ commutes over $\cX$ follows from the compatible family of isomorphisms $\hat{\xi}|_{\cZ_n} \cong \xi|_{\cW_n}$ and the Tannakian formalism (Black Box \ref{BB:tannakian}) as $\cX$ is a noetherian algebraic stack with affine stabilizers.) Finally, it is easy to see that $\xi \co \cW \to \cX$ is \'etale at $w$ which completes the proof.
\end{proof}

\subsection{Applications}  In this section, we provide a few applications of Theorems \ref{T:smooth} and \ref{T:general}.  We will not include the proofs and instead refer the reader to \cite{ahr}. 

\subsubsection{Application 1} \label{S:luna}
Our first application is a generalization of Luna's \'etale slice theorem and can be viewed as a refinement of Theorems \ref{T:smooth} and \ref{T:general} in the case that $\cX = [X/G]$ is a quotient stack.

\begin{theorem}\label{T:luna}
Let $X$ be a quasi-separated algebraic space, locally of finite type over $k$, with an action of a smooth affine group scheme $G$ over $k$.  Let $x \in X$ be a $k$-point with a linearly reductive stabilizer $G_x$.   Then there exist an affine scheme $W$ with an action of $G_x$ which fixes a point $w$ and an unramified $G_x$-equivariant morphism $(W,w) \to (X,x)$ such that $\tilde{f} \co W \times^{G_x} G \to X$ is \'etale.\footnote{Here, $W \times^{G_x} G$ denotes the quotient $(W \times G) / G_x$.  Note that there is an identification of GIT quotients  $(W \times^{G_x} G) \gitq G \cong W \gitq G_x$.}
 
Let $N_x = T_{X,x} / T_{G x, x}$ be the normal space to the orbit at $x$; this inherits a natural linear action of $G_x$.  If $x \in X$ is smooth, then it can be arranged that there is an \'etale $G_x$-equivariant morphism $W \to N_x$ such that $W \gitq G_x \to N_x \gitq G_x$ is \'etale and 
 $$\xymatrix{
N_x \times^{G_x} G \ar[d] &  W \times^{G_x} G\ar[r]^-{\tilde f} \ar[d] \ar[l]	& X \\
N_x \gitq G_x & W \gitq G_x \ar[l]	\ar@{}[ul]|\square				& 
}$$
is cartesian.
\end{theorem} 

\begin{remark}\label{R:luna}
The theorem above follows from Luna's \'etale slice theorem \cite{luna} if $X$ is affine. In this case, Luna's \'etale slice theorem is stronger than Theorem \ref{T:luna} as it asserts additionally that $W \to X$ can be arranged to be a locally closed immersion (which is obtained by choosing a $G_x$-equivariant section of $T_{X,x} \to N_x$ and then restricting to an open subscheme of the inverse image of $N_x$ under a $G_x$-equivariant \'etale morphism $X \to T_{X,x}$).  However, when $W$ is not normal, it is necessary to allow unramified neighborhoods (for instance, consider the example of $\GG_m$ acting on the nodal cubic). 
\end{remark}

\subsubsection{Application 2}
Our second application is a generalization of Sumihiro's theorem on torus actions.

\begin{theorem} \label{T:sumihiro}
Let $X$ be a quasi-separated algebraic space, locally of finite type over $k$, with an action of a torus $T$. If $x \in X$ is a $k$-point, then there exist a $T$-equivariant \'etale neighborhood $(\Spec A, u) \to (X,x)$.
\end{theorem}

\begin{remark}
More generally, if $X$ is a Deligne--Mumford stack (with the other hypotheses remaining the same), one can show that there exist a reparameterization $\alpha \co T \to T$ and an \'etale neighborhood $(\Spec A, u) \to (X,x)$ that is equivariant with respect to $\alpha$.
\end{remark}

In the case that $X$ is a normal scheme, Theorem \ref{T:sumihiro} was proved by Sumihiro in \cite[Cor.~2]{sumihiro1} in which case $\Spec A \to X$ can be taken to be an open neighborhood.  In the example of the $\GG_m$-action on the nodal cubic, there does not exist a $\GG_m$-invariant affine open neighborhood of the node.  Thus, for non-normal schemes, it is necessary in general to allow \'etale neighborhoods.

In fact,  we can prove more generally:

\begin{theorem} \label{T:sumi3}
Let $X$ be a quasi-separated algebraic space, locally of finite type over $k$, with an action of an affine group scheme $G$ of finite type over $k$.  Let $x \in X$ be a $k$-point with linearly reductive stabilizer $G_x$.  Then there exist an affine scheme $W$ with an action of $G$ and a $G$-equivariant \'etale neighborhood $W \to X$ of~$x$.
\end{theorem}

\subsubsection{Application 3}  We now translate the conclusion of Theorem \ref{T:general} in the case that $\cX$ is the stack of all (possibly singular) curves.  By a {\it curve}, we mean a proper scheme over $k$ of pure dimension one.
 
 \begin{theorem} \label{T:curves}
 Let $C$ be an $n$-pointed curve.  Suppose that every connected component of $C$ is either reduced of arithmetic genus $g \ne 1$ or contains a marked point. Suppose that $\Aut(C)$ is smooth and linearly reductive.  Then there exist an affine scheme $W$ of finite type over $k$ with an action of $\Aut(C)$ fixing a $k$-point $w \in W$
 and a miniversal deformation
 $$\xymatrix{
 \cC \ar[d]	 				&  C \ar[l] \ar[d]\\
 W 	\ar@{}[ur]|\square		&  \Spec k \ar[l]_{w}
 }$$
 of $C \cong \cC_w$ such that there exists an 
 action of $\Aut(C)$ on the total family $\cC$ compatible with the action of $\Aut(C)$ on $W$
 and $\cC_w$.
 \end{theorem}

\subsubsection{Application 4}  In our final application, we will conclude that under suitable hypotheses the coherent completion of a point of an algebraic stack exists and moreover we will give equivalent conditions for algebraic stacks to be \'etale locally isomorphic.

Let $\cX$ be a noetherian algebraic stack  over $k$ with affine stabilizers and let $x \in \cX$ be a closed $k$-point.  We say that $(\cX, x)$ is a {\it complete local stack} if 
the natural functor
$$\Coh(\cX)  \to  \ilim_n \Coh\bigl(\cX_{n}\bigr)$$
is an equivalence.  

\begin{theorem} \label{T:completions}
Let $\cX$ be a quasi-separated algebraic stack, locally of finite type over $k$, with affine stabilizers. For any $k$-point $x \in \cX$ with linearly reductive stabilizer $G_x$, there exist a complete local stack $(\hat{\cX}_x,\hat{x})$ and a morphism $\eta\colon (\hat{\cX}_x,\hat{x}) \to (\cX,x)$ inducing isomorphisms of $n$th infinitesimal neighborhoods of $\hat{x}$ and $x$. The pair $(\hat{\cX}_x,\eta)$ is unique up to unique $2$-isomorphism.
\end{theorem}

If $(\cW=[\Spec A/G_x], w)\to (\cX,x)$ is an \'etale morphism as in Theorem~\ref{T:general}
and $\phi\colon \cW\to W= \Spec A^{G_x}$ is the morphism to the GIT quotient, then the completion $\hat{\cX}_x$ is constructed as the fiber product
$$\xymatrix{
\hat{\cX}_x \ar[r] \ar[d]	& \cW \ar[d]^{\phi} \\
\Spec \hat{\oh}_{W, \phi(w)} \ar[r] \ar@{}[ur]|\square	& W . 
}$$

\begin{example}
The pair $([\AA^n / \GG_m], 0)$ is a complete local stack where $\GG_m$ acts with weight $1$ on all coordinates.  
\end{example}

\begin{example}
If we let $\cX = [\AA^2 / \GG_m]$ where $\GG_m$ acts with weights $(1,-1)$ on the coordinates $\AA^2 = \Spec k[x,y]$ and $x \in \cX$ denotes the origin, then $(\cX,x)$ is not a complete local stack.  The completion of $\cX$ at $x$ is given by the fiber product
$$\xymatrix{
\hat{\cX}_x  \ar[r] \ar[d]		& [\AA^2 / \GG_m] \ar[d] \\
\Spec k[[xy]] \ar[r]\ar@{}[ur]|\square			& \Spec k[xy]
}$$
which is identified with the quotient stack $[\Spec(k[[xy]] \otimes_{k[xy]} k[x,y] ) / \GG_m]$.
\end{example}

The next result is a stacky generalization of Corollary \ref{C:etale-local}, which was our first application of Artin approximation. 

\begin{theorem} \label{T:etale}
  Let $\cX$ and $\cY$ be quasi-separated algebraic stacks, locally of finite type over $k$, with affine stabilizers.  Suppose $x \in \cX$ and $y \in \cY$ are $k$-points with linearly reductive stabilizer group schemes $G_x$ and $G_y$, respectively.  Then the following are equivalent:
  \begin{enumerate}
  	\item
      There exist compatible isomorphisms $\cX_n \to \cY_n$.
  	\item
	  There exists an isomorphism $\hat{\cX}_x \to \hat{\cY}_y$.
  	\item
	  There exist an affine scheme $\Spec A$ with an action of $G_x$, a point $w \in \Spec A$ fixed by $G_x$, and a diagram of
      \'etale morphisms
	$$\xymatrix{	
			& [\Spec A /G_x] \ar[ld]_f \ar[rd]^g \\
		\cX	&  & \cY
	}$$
	such that $f(w) = x$ and $g(w) = y$, and both $f$ and $g$ induce isomorphisms of stabilizer groups at $w$.
  \end{enumerate}
If, in addition, the points $x \in \cX$ and  $y \in \cY$ are smooth and if the stabilizers $G_x$ and $G_y$ are smooth, then the conditions above are equivalent to the existence of an isomorphism $G_x \to G_y$ of group schemes and an isomorphism $T_{\cX,x} \to T_{\cY,y}$ of tangent spaces which is equivariant under $G_x \to G_y$.
\end{theorem}

\bibliography{refs}
\bibliographystyle{dary}

\end{document}